\documentclass[12pt]{amsart}
\usepackage{amsmath}
\usepackage{amsthm}
\usepackage{mathtools}
\usepackage{amsfonts}
\usepackage{amssymb}
\usepackage{mathrsfs}
\usepackage{hyperref}

\usepackage{anysize}
\marginsize{1in}{1in}{1.3in}{1in}

\usepackage{enumerate}

\numberwithin{equation}{section}
\newtheorem{theorem}{Theorem}[section]
\newtheorem{corollary}[theorem]{Corollary}
\newtheorem{lemma}[theorem]{Lemma}
\newtheorem{problem}[theorem]{Problem}
\newtheorem{proposition}[theorem]{Proposition}
\newtheorem{question}[theorem]{Question}
\newtheorem{remark}[theorem]{Remark}

\theoremstyle{definition}
\newtheorem{definition}[theorem]{Definition}
\newtheorem{notation}[theorem]{Notation}
\newtheorem{example}[theorem]{Example}

\makeatletter\renewenvironment{proof}[1][\proofname] {\par\pushQED{\qed}\normalfont\topsep6\p@\@plus6\p@\relax\trivlist\item[\hskip\labelsep\bfseries#1\@addpunct{.}]\ignorespaces}{\popQED\endtrivlist}

\def\A{\mathcal{A}}
\def\a{\mathbf a}
\def\al{\alpha}
\def\BC{\mathcal{B}}

\def\be{\beta}
\def\C{\mathbb{C}}
\def\D{\mathbf D}
\def\diag{\textnormal{diag}}
\def\e{\textnormal{e}}
\def\em{\varnothing}
\def\F{\mathbb F}
\def\g{\mathbf g}
\def\la{\lambda}
\def\N{\mathbb{N}}
\def\O{\mathcal O}
\def\R{\mathbb{R}}

\def\U{\mathscr{U}}

\def\Var{\textnormal{\bf Var}}
\def\Z{\mathbb{Z}}

\renewcommand\bar{\overline}

\renewcommand{\i}{\textnormal{i}}
\renewcommand\Re{\textnormal{Re}}
\renewcommand\phi{\varphi}

\begin{document}
\title{$*$-Freeness in Finite Tensor Products}

\author{Benoit Collins}
\address{Department of Mathematics,
Graduate School of Science,
Kyoto University,
Kyoto 606-8502,
Japan}
\email{collins@math.kyoto-u.ac.jp}

\author{Pierre Yves Gaudreau Lamarre}
\address{Department of Operations Research and Financial Engineering,
Princeton University,
Sherrerd Hall, Charlton Street
Princeton, NJ 08544}
\email{plamarre@princeton.edu}

\begin{abstract}
In this paper, we consider the following question and variants thereof: 
given $\D:=\big(a_{1;i}\otimes\cdots\otimes a_{K;i}:i\in I\big)$,
a collection of elementary tensor non-commutative random variables in the tensor product 
of
probability spaces $(\A_1\otimes\cdots\otimes\A_K,\phi_1\otimes\cdots\otimes\phi_K)$,
when is $\D$ $*$-free?
(See Section \ref{subsec:intro-notation} for a precise formulation of this problem.)

Settling whether or not freeness occurs in tensor products is a recurring problem in operator algebras,
and the following two examples provide a natural motivation for the above question:
\begin{enumerate}[(A)]
\item If $(a_{1;i}:i\in I)$ is a $*$-free family of Haar unitary variables and
$a_{k,i}$ are arbitrary unitary variables for $k\geq2$,
then the $*$-freeness persists at the level of the tensor product $\D$.
\item A converse of (A) holds true if all variables $a_{k;i}$ are group-like elements (see Corollary \ref{Corollary: Main Group} of Proposition \ref{Proposition: Main Group}).
\end{enumerate}
It is therefore natural to seek to understand the extent to which such simple characterizations hold true in more general cases.
While our results fall short of a complete characterization,
we make notable steps toward identifying necessary and sufficient conditions for the freeness of $\D$.
For example, we show that under evident assumptions, 
if more than one family $(a_{k,i}:i\in I)$ contains non-unitary variables,
then the tensor family fails to be $*$-free (see Theorem \ref{Theorem: Main Theorem} (1)). 
\end{abstract}

\maketitle


\section{Introduction}

\subsection{Motivating observations}

In connection with recent investigations on quantum expanders and related topics in operator algebras
\cite{OzawaPisier2016,Pisier2014,Pisier2015},
several years ago
G. Pisier and R. Speicher asked the following question to the first named author:
given $U_1^{(N)},\ldots, U_n^{(N)}$, $n$ independent Haar distributed $N\times N$ unitary random matrices,
are $$U_1^{(N)}\otimes \overline{U_1^{(N)}} ,\ldots, U_n^{(N)}\otimes \overline{U_n^{(N)}}$$
asymptotically $*$-free as $N\to\infty$ ($\bar{U_i^{(N)}}$ denotes the entrywise complex conjugate)?

Thanks to the almost sure asymptotic $*$-freeness of $(U_1^{(N)},\ldots, U_n^{(N)})$ (see \cite{HiaiPetz2000}, for instance),
a simple argument shows that the above question can be answered in the affirmative: 
for any collection $(V_1^{(N)},\ldots, V_n^{(N)})$ that converges in joint $*$-distribution to a collection $(v_1,\ldots,v_n)$ of (not necessarily $*$-free)
unitary variables in an arbitrary non-commutative $*$-probability space, 
$(U_1^{(N)}\otimes V_1^{(N)},\ldots, U_n^{(N)}\otimes V_n^{(N)})$ is almost surely asymptotically $*$-free.
This follows directly from the definition of asymptotic $*$-freeness and the fact that $(U_1^{(N)},\ldots, U_n^{(N)})$ converges to a collection of $*$-free Haar unitaries
(see Proposition \ref{Proposition: Tensor Freeness Conditions Implies Freeness} and Section \ref{Sufficient Conditions}
for a detailed proof of a more general version of this result).
Then,
taking $V_i^{(N)}=\overline{U_i^{(N)}}$ solves the above question,
and a version in expectation can be achieved thanks to the Dominated Convergence Theorem.

\begin{remark}
Note that the above reasoning cannot be used to similarly
extend results of strong asymptotic freeness
of random matrices (such as \cite{CollinsMale2014,Male2012})
to the strong asymptotic freeness of tensor products of random matrices.
Characterizing the occurence of strong convergence in tensor products
remains an unsolved and seemingly difficult problem.
\end{remark}

While taking the $V_i^{(N)}$ to be unitary
is natural given the present applications in operator algebras,
one may wonder if a similar phenomenon occurs when the unitarity assumption is dropped.
Understanding the mechanisms that give rise to $*$-freeness in general tensor products
turns out to be a very interesting and surprisingly difficult question with connections to group theory, 
which is what we explore in this paper.

\subsection{Main Problem}\label{subsec:intro-notation}

For a fixed $K\in\N$,
let $(\A_1,\phi_1),\ldots,(\A_K,\phi_K)$ be $*$-probability spaces,
and let $(\A,\phi)=(\A_1\otimes\cdots\otimes\A_K,\phi_1\otimes\cdots\otimes\phi_K)$
be their tensor product
(in which the $\A_k$ are independent in the classical probability sense).
For each $k\in \{1, \ldots , K\}$,
let $\a_k=(a_{k;i}:i\in I)\subset(\A_k,\phi_k)$
be a collection of random variables,
where the same indexing set $I$ is used for all $k$.
Consider the collection
$$\D=\diag(\a_1\otimes\cdots\otimes\a_K)=\big(a_{1;i}\otimes\cdots\otimes a_{K;i}:i\in I\big)\subset(\A,\phi),$$
that is,
the collection of tensor products $a_{1;i(1)}\otimes\cdots\otimes a_{K;i(K)}$ such that
$i(1)=\cdots=i(K)$.
Specifically, if $I=\{1,\ldots ,n\}$ is a finite set (as we will always be able to
assume without loss of assumption), 
$$\D=(a_{1;1}\otimes\cdots\otimes a_{K;1}\,\,,\,\,
a_{1;2}\otimes\cdots\otimes a_{K;2}\,\,,\,\,
\ldots \,\,,\,\, a_{1;n}\otimes\cdots\otimes a_{K;n}).$$
The problem we investigate in this paper is the following.

\begin{problem}\label{Problem: General Problem}
When is the collection $\D$ $*$-free?
\end{problem}

In order to formulate our results concerning Problem \ref{Problem: General Problem},
we introduce two definitions.

\begin{definition}
A noncommutative polynomial $M\in\C\langle x,x^*\rangle$ in the indeterminates $x$ and $x^*$
is called a $\boldsymbol*${\bf-word}
if it can be written as
$$M(x)=x^{n(1)}\cdots x^{n(t)},\qquad\text{$t\in\N$ and $n(1),\ldots,n(t)\in\{1,*\}$},$$
that is,
$M$ is a monomial with no constant factor.
\end{definition}

\begin{definition}\label{Definition: Tensor Freeness Conditions}
We say that $\D$ satisfies the {\bf tensor freeness conditions} (TFC) if:
there exists $k\in \{1, \ldots , K\}$ such that the collection $\a_k$ is $*$-free,
and such that for every $*$-word $M$ and index $i\in I$,
\begin{enumerate}
\item if $\phi\big(M(a_{1;i}\otimes\cdots\otimes a_{K;i})\big)=0,$
then $\phi_k\big(M(a_{k;i})\big)=0$; and
\item if $\phi\big(M(a_{1;i}\otimes\cdots\otimes a_{K;i})\big)\neq0,$
then $M(a_{l;i})$ is deterministic (i.e.,
a constant multiple of the unit vector in $\A_l$) for every $l\neq k$.
\end{enumerate}
In this case, we call 
$\a_k$
a {\bf dominating collection}.
\end{definition}

It can be shown with straightforward computations that the TFC provide a sufficient condition for Problem \ref{Problem: General Problem}:

\begin{proposition}\label{Proposition: Tensor Freeness Conditions Implies Freeness}
If $\D$ satisfies the TFC,
then it is $*$-free.
\end{proposition}

Indeed,
as explained in greater detail in Section \ref{Sufficient Conditions},
the TFC are arguably one of the simplest sufficient conditions for $\D$ to be $*$-free:
we assume that one of the collections $\a_k$ is $*$-free,
and then conditions (1) and (2) in Definition \ref{Definition: Tensor Freeness Conditions}
are specifically designed to
ensure that the $*$-freeness present in $\a_k$ will be preserved in $\D$.

As one might expect,
the TFC do not characterize the $*$-freeness of $\D$ in general,
as it is possible to construct an example where neither $\a_1$ nor $\a_2$ is $*$-free
and $\D=\diag(\a_1\otimes\a_2)$ is $*$-free,
and thus with no dominating collection
(see Example \ref{Example: Tensor Freeness Conditions Are Not Necessary}).

The main results of this paper are that,
in two specific situations,
the TFC are necessary for the $*$-freeness of $\D$,
namely, in the case were the elements $a_{k;i}$ (a) are group like elements
and (b) form $*$-free families.

In the forthcoming two subsections, we elaborate on these two cases (Section \ref{subsec:group-algebra}
for case (a), and Section \ref{subsec:*-free-family} for case (b)).

\subsection{Group Algebras}\label{subsec:group-algebra}

The first case we consider is that of group algebras equipped with the canonical trace.
In this setting,
we obtain the following result (Section \ref{Section: Groups}).

\begin{proposition}\label{Proposition: Main Group}
Let $G_1,\ldots,G_K$ be groups.
For every $k\in \{1, \ldots , K\}$,
let
$\g_k=(g_{k;i}:i\in I)\subset G_k,$
be a collection of group elements,
and define the direct product collection
$$\D_\times=\diag(\g_1\times\cdots\times\g_K)=\big((g_{1;i},\ldots,g_{K;i}):i\in I\big),$$
assuming that $\D_\times$ does not contain the neutral element.
If the collection $\D_\times$ is free,
then there exists $k\in \{1, \ldots , K\}$ such that
the collection $\g_k$ is free in $G_k$, and
for every $n\in\N$ and $i\in I$,
if $(g_{1;i}^n,\ldots,g_{K;i}^n)\neq e$,
then $g_{k,i}^n\neq e$.
\end{proposition}

Given the well-known correspondence between $*$-freeness in groups algebras equipped with the canonical trace
and freeness in groups (see Proposition \ref{Proposition: Freeness in Groups vs Group Algebras}),
the following corollary regarding the TFC is readily established.

\begin{corollary}\label{Corollary: Main Group}
Suppose that for every $k\in \{1, \ldots , K\}$,
$\A_k=\C G_k$ is the complex group algebra of some group $G_k$,
$\phi_k=\tau_e$ is the canonical trace,
and $\a_k\subset G_k$ is a collection of group elements.
Suppose further that $\D$ does not contain the unit vector in $\A$.
If $\D$ is $*$-free,
then the TFC are satisfied.
\end{corollary}

Proposition \ref{Proposition: Main Group} can chiefly be explained by the observation that the structure of the group
operation in a direct product can sometimes introduce nontrivial relations between elements
other than the identity through products of the form
\begin{align}\label{Equation: Nontrivial Relations in Direct Products}
e=(g,e)(e,h)(g,e)^{-1}(e,h)^{-1},\qquad g,h\neq e,
\end{align}
and that the TFC (or an equivalent formulation of the TFC in the group context)
are the only way to ensure that this phenomenon does not occur.

\subsection{$*$-Free Families}\label{subsec:*-free-family}

In the second case that we consider, we assume that the families $\a_1,\ldots,\a_K$
are $*$-free,
and that the functionals $\phi_1,\ldots,\phi_K$ are faithful traces.
In this setting,
we obtain the following necessary conditions (Section \ref{Section: Main Result}).

\begin{theorem}\label{Theorem: Main Theorem}
Suppose that for every $k\in \{1, \ldots , K\}$,
the collection $\a_k$ is $*$-free and $\phi_k$ is a faithful trace,
and suppose that $\D$ does not contain the zero vector or a scalar multiple of the unit vector in $\A$.
If $\D$ is $*$-free,
then the following conditions hold.
\begin{enumerate}
\item For all $k\in \{1, \ldots , K\}$ except at most one,
every variable in $\a_k$ is a constant multiple of a unitary variable.
\item If there exists $k\in \{1, \ldots , K\}$ such that $\a_k$ contains a variable that is not a constant multiple of a unitary variable,
then the TFC are satisfied with $\a_k$ as a dominating collection.
\item Suppose that $\a_1,\ldots,\a_k$ only contain constant multiples of unitary variables.
If there exists $k\in \{1, \ldots , K\}$,
a $*$-word $M\in\C\langle x,x^*\rangle$, and $i\in I$ such that $\phi\big(M(a_{1;i})\otimes\cdots\otimes M(a_{K;i})\big)\neq0$ and $M(a_{k;i})\neq\phi_k\big(M(a_{k;i})\big)$,
then $\D$ satisfies the TFC and $\a_k$ is a dominating collection.
\end{enumerate}
\end{theorem}

\begin{remark}
Note that Theorem \ref{Theorem: Main Theorem} falls one case short of providing a complete characterization for the $*$-freeness of $\D$
when the $\a_k$ are $*$-free and the $\phi_k$ are faithful traces.
Indeed,
the only case where it is not shown that the TFC are necessary for the $*$-freeness of $\D$
is the one where $\a_1,\ldots,\a_k$ only contain constant multiples of unitary random variables
and condition (3) in Theorem \ref{Theorem: Main Theorem} does not hold.
Unfortunately,
it can be shown that the methods we use in this paper cannot settle this remaining case (see Section \ref{Section: Discussion} for more details).
\end{remark}

Our method of proof for Theorem \ref{Theorem: Main Theorem} uses elementary methods:
by definition of tensor product of $*$-probability spaces,
\begin{align}\label{Equation: Main Theorem Method}
\phi\big(M(a_{1;i}\otimes\cdots\otimes a_{K;i}:i\in I)\big)=\phi_1\big(M(a_{1;i}:i\in I)\big)\cdots\phi_K\big(M(a_{K;i}:i\in I)\big)
\end{align}
for any $*$-word $M\in\C\big\langle x_i,x_i^*:i\in I\rangle$.
If it is assumed that $\D$ is $*$-free and that $\a_k$ is $*$-free for all $k\in \{1, \ldots , K\}$,
then by using free probability,
it is possible to reduce both sides of (\ref{Equation: Main Theorem Method}) as quantities that only depend on the distributions
of the variables $a_{k;i}$,
and thus obtain necessary conditions for the $*$-freeness (see Section \ref{Section: General Strategy of Proof} for more details).
The assumption that the functionals $\phi_k$ are 
faithful is
used to obtain various Cauchy-Schwarz-type inequalities (such as Proposition \ref{Proposition: Cauchy-Schwarz}),
which allow us to derive multiple technical estimates (Appendix \ref{Appendix}) that are crucial in our proof.

\subsection{Organization}
The organization of this paper is as follows.
In Section \ref{Background in Free Probability},
we provide the notions in free probability that are used in this paper.
In Section \ref{Sufficient Conditions},
we prove that the TFC are sufficient for the $*$-freeness of $\D$ (Proposition \ref{Proposition: Tensor Freeness Conditions Implies Freeness}),
and we provide an example that shows that the TFC are not necessary in general.
In Section \ref{Section: Groups},
we prove Proposition \ref{Proposition: Main Group}.
In Section \ref{Section: Main Result},
we prove Theorem \ref{Theorem: Main Theorem}.
In Section \ref{Section: Discussion},
we discuss the limitations of our methods and possible directions for future research.
In Appendix \ref{Appendix},
we prove a few technical results that are used in some of our proofs but that are otherwise unrelated to the subject of this paper.

\subsection{Acknowledgements}
The research reported in this paper was initiated during the M.Sc. study of 
P.-Y. G.-L. under the supervision of B.C. at the University of Ottawa. 
The authors thank Camille Male for constructive discussions regarding early drafts of this paper as well as his significant involvement 
in P.-Y. G.-L.'s studies at that time.
The authors thank Yves de Cornulier for insightful comments regarding Proposition 1.6 and his answer to the MathOverflow question \href{http://www.mathoverflow.net/questions/204370/}{mathoverflow.net/questions/204370}.
They also thank an anonymous referee for a very careful reading of the initial version of the manuscript and many 
constructive suggestions that improved the paper.
B.C. was supported financially by NSERC, JSPS Kakenhi,  and ANR-14-CE25-0003.
P.-Y. G.-L. was supported financially during his M.Sc. studies by NSERC and OGS scholarships,
and as a Ph.D. student by NSERC and Gordon Wu scholarships.


\section{Background in Free Probability}\label{Background in Free Probability}

In this section,
we introduce the notation and results in free probability that are used in this paper.
For a more thorough introduction to the subject,
the reader is referred to
\cite{NicaSpeicher2006} Lectures 1 to 5, or \cite{VoiculescuDykemaNica1992} Chapters 1 and 2.

\subsection{$*$-Probability Spaces}

A $\boldsymbol*${\bf-probability space} consists of a pair $(\BC,\psi)$,
where $\BC$ is a unital $*$-algebra over $\C$,
and $\psi:\BC\to\C$ is a linear functional that is unital (i.e., $\psi(1)=1$) and positive (i.e., $\psi(bb^*)\geq0$ for all $b\in\BC$).
The following is a fundamental example of $*$-probability space.

\begin{definition}\label{Definition: Group Algebras}(\cite{NicaSpeicher2006} Example 1.4.)
Given a group $G$,
let $\C G$ denote its group algebra over $\C$,
and let $\tau_e:\C G\to\C$ denote the {\bf canonical trace},
which is defined as the linear extension to $\C G$ of the following map on $G$:
$$
\tau_e(g)=
\begin{cases}
1&\text{if }g=e;\text{ and}\\
0&\text{otherwise.}
\end{cases}
$$
$(\C G,\tau_e)$ is a $*$-probability space.
\end{definition}

Let $(\BC,\psi)$ be a $*$-probability space.
The functional $\psi$ is said to be {\bf faithful} if $\psi(bb^*)=0$ if and only if $b=0$,
and a {\bf trace} if $\psi(bc)=\psi(cb)$ for all $b,c\in\BC$.
An element $b\in\A$ is said to be {\bf unitary} if $bb^*=b^*b=1$.
The {\bf variance}
is defined as
$$\Var[b]:=\psi\Big(\big(b-\psi(b)\big)\big(b-\psi(b)\big)^*\Big).$$

\begin{proposition}\label{Proposition: Variance with Faithful Functionals}
Let $(\BC,\psi)$ be a $*$-probability space such that $\psi$ is faithful.
For every $b\in\BC$,
if $\Var[b]=0$,
then $b=\psi(b)$ (i.e.,
$b$ is a constant multiple of the unit vector,
in other words, $b$ is deterministic).
\end{proposition}

The following elementary propositions will be useful in computations in later sections.

\begin{proposition}\label{Proposition: Cauchy-Schwarz}
Let $b$ be a noncommutative random variable in a $*$-probability space $(\BC,\psi)$.
If $\psi(bb^*)=1$,
then
\begin{align}\label{Equation: Norm of Element Condition}
|\psi(b)|\leq 1\leq\psi\big((bb^*)^2\big).
\end{align} 
\end{proposition}
\begin{proof}
Since $\psi$ is positive,
$0\leq\psi\big((\e^{\i\theta}+b)(\e^{-\i\theta}+b^*)\big)=2+2\Re\big(\e^{-\i\theta}\psi(b)\big)$
for every $\theta\in[0,2\pi)$.
By choosing $\theta$ such that $\Re(\e^{-\i\theta}\psi(b))=-|\psi(b)|$,
this yields $|\psi(b)|\leq 1$.
Furthermore,
$0\leq\psi\big((bb^*-1)(bb^*-1)\big)=\psi\big((bb^*)^2\big)-1,$
hence $\psi\big((bb^*)^2\big)\geq1$.
\end{proof}

\begin{proposition}\label{Proposition: Unitary Equal to Unit}
Let $(\BC,\psi)$ be a $*$-probability space such that $\psi$ is faithful,
and let $b\in\BC$ be such that $\psi(bb^*)=1$.
If $|\psi(b)|=1$,
then $b=\psi(b)$.
Furthermore,
if $b=\la$ for some $\la\in\C$,
then $|\la|=1$.
\end{proposition}
\begin{proof}
If $\psi(bb^*)=1$,
then
$\Var[b]=\psi(bb^*)-\psi(b)\psi(b^*)=1-|\psi(b)|^2,$
and the result follows from Proposition \ref{Proposition: Variance with Faithful Functionals}.
\end{proof}

\subsection{$*$-freeness}

Let $(\BC,\psi)$ be a $*$-probability space.
A collection $(\BC_i:i\in I)$ of unital $*$-subalgebras of $\BC$ is said to be $\boldsymbol *${\bf -free}
if for every $t\in\N$ and random variables $b_1\in\BC_{i(1)},\ldots,b_t\in\BC_{i(t)}$,
one has $\psi(b_1\cdots b_t)=0$ whenever $i(1)\neq i(2)\neq \cdots\neq i(t)$,
and
$\phi(b_1)=\cdots=\phi(b_t)=0$.
In this context,
a collection $(b_i:i\in I)$ of noncommutative random variables in $\BC$ is said to be $\boldsymbol *${\bf-free}
if the collection of unital $*$-algebras the $b_i$ generate is $*$-free.

\begin{definition}\label{Definition: Centering}
Given a noncommutative random variable $b$ in a $*$-probability space $(\BC,\psi)$,
we use $b^\circ$ to denote the {\bf centering} of $b$,
that is,
$b^\circ=b-\psi(b)$,
hence $\psi(b^\circ)=0$.
\end{definition}

\begin{remark}\label{Remark: Centering and Freeness}
By linearity,
a collection $(b_i:i\in I)\subset\BC$ of noncommutative random variables is $*$-free if and only if for every
$t\in\N$,
indices $i(1)\neq\cdots\neq i(t)$,
and $*$-words $M_1,\ldots,M_t\in\C\langle x,x^*\rangle$,
one has $\phi\big(M_1(b_{i(1)})^\circ\cdots M_t(b_{i(t)})^\circ\big)=0.$
\end{remark}

There exists a well-known correspondence between $*$-freeness in group algebras with the canonical trace
and freeness in groups,
which we state in the following proposition.

\begin{proposition}[\cite{NicaSpeicher2006} Proposition 5.11]\label{Proposition: Freeness in Groups vs Group Algebras}
Let $(g_i:i\in I)$ be a collection of elements in a group $G$.
$(g_i:i\in I)$ is $*$-free as a collection of random variables in $(\C G,\tau_e)$ if and only if $(g_i:i\in I)$ is free
in the group $G$.
\end{proposition}

Recall that a collection $(g_i:i\in I)\subset G$ of group elements is said to be {\bf free} if
the canonical homomorphism from the free product of cyclic groups $*_{i\in I}\langle g_i\rangle$
to the group $\langle g_i:i\in I\rangle$ generated by the $g_i$
is an isomorphism,
or,
in other words,
if $g_{i(1)}^{n(1)}\cdots g_{i(t)}^{n(t)}\neq e$ whenever $i(1)\neq\cdots\neq i(t)$,
and $n(1),\ldots,n(t)\in\Z$ are such that $g_{i(l)}^{n(l)}\neq e$ $(l\leq t)$.

The concept 
$*$-freeness can be thought of as an analog of independence for commutative $\C$-valued random variables.
Indeed,
if $(\BC_i:i\in I)$ is $*$-free,
then any expression of the form
$$\psi(b_1\cdots b_t),\qquad t\in\N,~b_1\in\BC_{i(1)},\ldots,b_t\in\BC_{i(t)}$$
can
in principle
be computed using the restriction of $\psi$ onto the unital $*$-subalgebras $\BC_i$.
The following examples of this phenomenon will be used extensively in this paper.

\begin{proposition}[\cite{NicaSpeicher2006} Example 5.15 (3)]\label{Proposition: Computations}
Let $(\BC,\psi)$ be a $*$-probability space,
and let $(\BC_1,\BC_2)$ be $*$-free.
If $b_1\in\BC_1$ and $b_2\in\BC_2$,
then
\begin{align}\label{Equation: *-freeness Computation 1}
\psi(b_1b_2b_1^*b_2^*)=|\psi(b_1)|^2\psi(b_2b_2^*)+|\psi(b_2)|^2\psi(b_1b_1^*)-|\psi(b_1)|^2|\psi(b_2)|^2.
\end{align}
\end{proposition}

\begin{remark}
Let $(\BC,\psi)$ be a $*$-probability space,
and let $(\BC_1,\BC_2)$ be $*$-free.
If $b\in\BC_1$ is unitary and such that $\psi(b)=0$,
then it can easily be shown that $(b\BC_2b^*,\BC_2)$ is $*$-free.
Thus,
it follows from (\ref{Equation: *-freeness Computation 1}) that for every $c_1,c_2\in\BC_2$,
one has
\begin{align}\label{Equation: *-freeness Computation 2}
\psi(bc_1b^*c_2bc_1^*b^*c_2^*)=|\psi(c_1)|^2\psi(c_2c_2^*)+|\psi(c_2)|^2\psi(c_1c_1^*)-|\psi(c_1)|^2|\psi(c_2)|^2.
\end{align}
\end{remark}
Interestingly, the above elementary remark seems not to be available in the literature,
unless $b$ is assumed to be Haar unitary.


\section{Proof of Proposition \ref{Proposition: Tensor Freeness Conditions Implies Freeness}}\label{Sufficient Conditions}

In this section,
we prove that the TFC are sufficient for $\D$ to be $*$-free.
We begin with some notation to improve readability.

\begin{notation}\label{Notation: D_i}
For every $i\in I$,
let $D_i:=a_{1;i}\otimes\cdots\otimes a_{K;i}$,
that is,
$\D=(D_i:i\in I)$.
\end{notation}

As explained in the introduction to this paper,
the TFC are designed in such a way that the $*$-freeness present in a dominating collection is extended to the tensor product
collection.
We now explain this reasoning.
By definition (see Remark \ref{Remark: Centering and Freeness}),
$\D$ is $*$-free if and only if 
for every $t\in\N$,
indices $i(1)\neq\cdots\neq i(t)$,
and $*$-words $M_1,\ldots,M_t$,
one has
\begin{align}\label{Equation: Sufficient Condition eq1}
\phi\big(M_1(D_{i(1)})^\circ\cdots M_t(D_{i(t)})^\circ\big)=0.
\end{align}
Suppose that the collection $\a_1$ is $*$-free,
and that we wish to find simple conditions such that the $*$-freeness of $\a_1$ induces the $*$-freeness of $\D$.
If for every $*$-word $M$ and $i\in I$,
\begin{align}\label{Equation: Sufficient Condition eq2}
M(D_i)^\circ=M(a_{1;i})^\circ\otimes M(a_{2;i})\otimes\cdots\otimes M(a_{K;i}),
\end{align}
then
\begin{align*}
&\phi\big(M_1(D_{i(1)})^\circ\cdots M_t(D_{i(t)})^\circ\big)\\
&\qquad=\phi\left(\prod_{s=1}^t\Big(M_s(a_{1;i(s)})^\circ\otimes M_s(a_{2;i(s)})\otimes\cdots\otimes M_s(a_{K;i(s)})\Big)\right)\\
&\qquad=\phi_1\big(M_1(a_{1;i(1)})^\circ\cdots M_t(a_{1;i(t)})^\circ\big)\prod_{k=2}^K\phi_k\big(M_1(a_{k;i(1)})\cdots M_t(a_{k;i(t)})\big).
\end{align*}
Since $\a_1$ is $*$-free,
$\phi_1\big(M_1(a_{1;i(1)})^\circ\cdots M_t(a_{1;i(t)})^\circ\big)=0,$
which implies that (\ref{Equation: Sufficient Condition eq1}) holds.
Thus,
to complete the proof of Proposition \ref{Proposition: Tensor Freeness Conditions Implies Freeness},
it suffices to show that the TCF imply that (\ref{Equation: Sufficient Condition eq2}) holds.

\begin{proof}[Proof of Proposition \ref{Proposition: Tensor Freeness Conditions Implies Freeness}]
Suppose that $\D$ satisfies the TFC,
and assume without loss of generality that $\a_1$ is a dominating collection.
Let $i\in I$ and $M\in\C\langle x,x^*\rangle$ be a $*$-word.
We consider the two possible cases:
\begin{enumerate}
\item $\phi\big(M(D_i)\big)=0$; and
\item $\phi\big(M(D_i)\big)\neq0$.
\end{enumerate}
\noindent{\bf (1).}
Suppose that $\phi\big(M(D_i)\big)=0$.
Then,
$$M(D_i)^\circ=M(D_i)=M(a_{1;i})\otimes M(a_{2;i})\otimes\cdots\otimes M(a_{K;i}).$$
Since $\a_1$ is a dominating collection,
$\phi\big(M(D_i)\big)=0$ implies that $\phi_1\big(M(a_{1;i})\big)=0$,
i.e.,
$M(a_{1;i})^\circ=M(a_{1;i})$.
It is then clear that (\ref{Equation: Sufficient Condition eq2}) holds in this case.

\noindent{\bf (2).}
Suppose that $\phi\big(M(D_i)\big)\neq0$.
Since $\a_1$ is a dominating collection,
this implies that for every $k\geq 2$,
$M(a_{k;i})$ is deterministic,
hence equal to its expected value $\phi_k\big(M(a_{k;i})\big)$.
Therefore, $M(D_i)^\circ=M(D_i)-\phi\big(M(D_i)\big)$ is equal to
\begin{multline*}
M(a_{1;i})\otimes M(a_{2;i})\otimes\cdots\otimes M(a_{K;i})-\phi_1\big(M(a_{1;i})\big)\otimes M(a_{2;i})\otimes\cdots\otimes M(a_{K;i})\\
=M(a_{1;i})^\circ\otimes M(a_{2;i})\otimes\cdots\otimes M(a_{K;i}),
\end{multline*}
and thus (\ref{Equation: Sufficient Condition eq2}) also holds in this case,
concluding the proof of the proposition.
\end{proof}

\begin{remark}
If the algebras $\A_k$ are not domains (i.e., $ab=0$ implies that $a=0$ or $b=0$ for all $a,b\in\A_k$),
then the TFC and (\ref{Equation: Sufficient Condition eq1}) need not be equivalent.
However,
given our methods,
the TFC are more convenient to work with.
\end{remark}

As claimed in the introduction to this paper,
while
the TFC are sufficient for $\D$ to be $*$-free,
they are not necessary in general.
Indeed,
we can construct an example
where $\a_1$ and $\a_2$ are not $*$-free and $\D$ is $*$-free,
and thus without any dominating collection.

\begin{example}\label{Example: Tensor Freeness Conditions Are Not Necessary}
Let $\F_2=\langle g,h\rangle$ be the free group with two generators $g$ and $h$,
and let $(\Z,+)$ be the additive group on $\Z$.
Let $\A_1=\C\F_2$ and $\A_2=\C\Z$,
equip $\A_2$ with the canonical trace $\phi_2=\tau_e$,
and equip $\A_1$ with the linear extension of the map $\phi_1$ defined on $\F_2$ as follows:
\begin{align}\label{eq: phi1}
\phi_1(m)=\begin{cases}
1&\text{if }m=e;\\
\beta&\text{if $m=gh,(gh)^{-1}$; and}\\
0&\text{otherwise}
\end{cases}
\end{align}
where $\be\in\R\setminus\{0\}$.
Clearly,
$(\A_2,\phi_2)$ is a $*$-probability space.
Further,
it can be shown with straightforward computations that
if $\be$ is small enough,
then $\phi_1$ is positive and faithful on $\A_1$,
hence $(\A_1,\phi_1)$ is a $*$-probability space as well.

Since $(g,h)$ is free in $\F_2$ and of infinite order,
it follows that $\big((g,\bar{1}),(h,\bar{2})\big)$ is free in $\F_2\times\Z$
(we use $\bar{n}$ to denote integers in the additive group $(\Z,+)$
to avoid confusion with scalars in $\C$).
Therefore,
$(g\otimes\bar{1},h\otimes\bar{2})$ is $*$-free in the $*$-probability space $(\C\F_2\otimes\C\Z,\tau_e\otimes\tau_e)$
(by Proposition \ref{Proposition: Freeness in Groups vs Group Algebras}).

It can be noticed that,
when restricted to the $*$-algebra generated by $g\otimes\bar{1}$ and $h\otimes\bar{2}$,
the functional $\phi_1\otimes\phi_2$ is equal to the canonical trace,
and thus $(g\otimes\bar{1},h\otimes\bar{2})$ is $*$-free in $(\A_1\otimes\A_2,\phi_1\otimes\phi_2)$.
Therefore,
if we define $\a_1=(g,h)$ and $\a_2=(\bar{1},\bar{2})$,
then $\D=\diag(\a_1\otimes\a_2)=(g\otimes\bar{1},h\otimes\bar{2})$ is $*$-free in $(\A_1\otimes\A_2,\phi_1\otimes\phi_2)$.
However,
$\a_1$ is not $*$-free,
as $\phi_1(gh)=\beta\neq0$ while $\phi_1(g)=\phi_1(h)=0$;
and $\a_2$ is not $*$-free,
since $\phi_2(\bar{1}\cdot\bar{2}\cdot\bar{1}^*\cdot\bar{2}^*)=\phi_2(\bar{1}+\bar{2}-\bar{1}-\bar{2})=\phi_2(\bar{0})=1$,
yet $\phi_2(\bar{1})=\phi_2(\bar{2})=0$.
\end{example}

\begin{question}
One easily notices that the functional $\phi_2$ in example \ref{Example: Tensor Freeness Conditions Are Not Necessary} is not a trace.
As we were unable to find an example where the TFC fail and all the considered functionals are tracial,
it is natural to wonder if the TFC always hold with traces. 
We leave it as an open question.
\end{question}


\section{Proof of Proposition \ref{Proposition: Main Group}}\label{Section: Groups}

We begin with the following reduction.

\begin{lemma}\label{Lemma: Induction Step for Groups}
If Proposition \ref{Proposition: Main Group} holds for $K=2$,
then it also holds for any $K>2$.
\end{lemma}
\begin{proof}
We proceed by induction.
Suppose that the hypotheses of Proposition \ref{Proposition: Main Group} are met,
let $K>2$,
and suppose that Proposition \ref{Proposition: Main Group} holds for $2,3,\ldots,K-1$,
and that $\D_\times$ is free.
Write the elements in the collection $\D_\times$ as
$D_i=\big(g_{1;i},(g_{2;i},\ldots,g_{K;i})\big)\in G_1\times(G_2\times\cdots\times G_K).$
Since Proposition \ref{Proposition: Main Group} holds in the case $K=2$,
\begin{enumerate}
\item $\g_1$ is free,
and for every $n\in\N$ and $i\in I$,
if $\big(g_{1;i}^n,(g_{2;i},\ldots,g_{K;i})^n\big)\neq e$,
then $g_{1;i}^n\neq e$; or
\item the collection $\diag(\g_2\times\cdots\times\g_K)$ is free,
and for every $i\in I$ and $n\in\N$,
if $\big(g_{1;i}^n,(g_{2;i},\ldots,g_{K;i})^n\big)\neq e$,
then $(g_{2;i},\ldots,g_{K;i})^n\neq e$.
\end{enumerate}
If case $(1)$ holds,
then the result is proved.
Suppose that case $(2)$ holds.
Then,
since Proposition \ref{Proposition: Main Group} is true for $K-1$,
the freeness of $\diag(\g_2\times\cdots\times\g_K)$ implies that
there exists
an integer $2\leq k\in \{1, \ldots , K\}$ such that $\g_k$ is free; and
for every $n\in\N$ and $i\in I$,
if $(g_{2;i}^n,\ldots,g_{K;i}^n)\neq e$,
then $g_{k,i}^n\neq e$.
Given that $\big(g_{1;i}^n,(g_{2;i},\ldots,g_{K;i})^n\big)\neq e$ implies that $(g_{2;i},\ldots,g_{K;i})^n\neq e$ in the present case,
the proof is complete.
\end{proof}

Thus,
we need only prove Proposition \ref{Proposition: Main Group} for $K=2$.
As stated in the introduction to this paper
(more precisely,
equation (\ref{Equation: Nontrivial Relations in Direct Products})),
Proposition \ref{Proposition: Main Group} can be explained by the fact that
the group operation in a direct product can induce nontrivial relations
that prevent the freeness of $\D_\times$.
The next lemma makes this claim precise.
(As in the previous section,
we use the notation $D_i=(g_{1;i},g_{2;i})$.)

\begin{lemma}\label{Lemma: Direct Product Induces Relations}
Suppose that $\D_\times$ does not contain the neutral element,
and that there exists $m,n\in\N$ and $i\in I$ such that
$g_{1;i}^m=e\neq g_{2;i}^m$ and $g_{1;i}^n\neq e=g_{2;i}^n$.
Then,
$\D_\times$ is not free.
\end{lemma}
\begin{proof}
Let $j\in I$ be such that $i\neq j$.
Since $D_i^n,D_i^m,D_{j}\neq e$,
if $\D_\times$ is free,
then
$D_i^mD_{j} D_i^nD_{j}^{-1}D_i^{-m}D_{j} D_i^{-n}D_{j}^{-1}\neq e$.
A direct computation shows that this is not the case,
as this expression reduces to the identity by definition of the group operation in a direct product.
Therefore,
$\D_\times$ is not free.
\end{proof}

Let
$F=\big\langle (g_{1;i},g_{2;i}):i\in I\big\rangle$ be the group generated by the collection $\D_\times$.
If we assume that $\D_\times$ is free,
then $F$ is canonically isomorphic to the free product of cyclic groups
$*_{i\in I}\big\langle(g_{1;i},g_{2;i})\big\rangle.$
For $k=1,2$,
let $\pi_k$ be the projection homomorphism from $F$ to the group $G_k$,
that is,
$\pi_k\big((g_1,g_2)\big)=g_k$.
Since $\ker(\pi_1)\subset\{e\}\times G_2$ and $\ker(\pi_2)\subset G_1\times\{e\},$
it follows that $\ker(\pi_1)$ and $\ker(\pi_2)$ commute and that $\ker(\pi_1)\cap\ker(\pi_2)=\{e\}$.
The next lemma provides a convenient way of establishing Proposition \ref{Proposition: Main Group} using projection homomorphisms.

\begin{lemma}\label{Lemma: Kernel Criterion}
Suppose that $\D_\times$ is free.
If there exists $k\in\{1,2\}$ such that $\ker(\pi_k)=\{e\}$,
then $\g_k$ is free,
and for every $n\in\N$ and $i\in I$,
if $(g_{1;i}^n,g_{2;i}^n)\neq e$,
then $g_{k;i}^n\neq e$.
\end{lemma}
\begin{proof}
Suppose that $\D_\times$ is free,
and that $\ker(\pi_k)=\{e\}$ for some $k\in\{1,2\}$.
If $(g_{1;i}^n,g_{2;i}^n)\neq e$,
then $(g_{1;i}^n,g_{2;i}^n)\not\in\ker(\pi_k)$,
hence $g_{k;i}^n\neq e$.
Thus,
it only remains to prove that $\g_k$ is free.
Let $i(1)\neq i(2)\neq \cdots\neq i(t)$ and $n(1),\ldots,n(t)\in\Z$ be such that
$g_{k;i(l)}^{n(l)}\neq e$ for all $l\leq t$.
Since $D_{i(l)}^{n(l)}\neq e$ for all $l\leq t$,
the freeness of $\D_\times$ implies that $D_{i(1)}^{n(1)}\cdots D_{i(t)}^{n(t)}\neq e$.
Thus,
$D_{i(1)}^{n(1)}\cdots D_{i(t)}^{n(t)}\not\in\ker(\pi_k)$,
which implies that $g_{k;i(1)}^{n(1)}\cdots g_{k;i(t)}^{n(t)}\neq e$,
and hence $\g_k$ is free.
\end{proof}

\begin{proof}[Proof of Proposition \ref{Proposition: Main Group}]
Assume for contradiction that $\ker(\pi_k)\neq\{e\}$ for both $k=1,2$.
Then,
$\ker(\pi_1)\not\subset\ker(\pi_2)$ and $\ker(\pi_2)\not\subset\ker(\pi_1)$,
otherwise $\ker(\pi_1)\cap\ker(\pi_2)\neq\{e\}$.
Thus,
there exists $x,y\neq e$ such that
$x\in\ker(\pi_1)\setminus\ker(\pi_2)$ and $y\in\ker(\pi_2)\setminus\ker(\pi_1)$.
Since $\ker(\pi_1)$ and $\ker(\pi_2)$ commute,
$x$ and $y$ commute in a free product.
Therefore,
$x,y\in\langle w\rangle$ for some $w\in F$,
or $x,y\in z\big\langle (g_{1;i},g_{2;i})\big\rangle z^{-1}$ for some $z\in F$ and $i\in I$ (see \cite{MagnusKarrasSolitar1966} Corollary 4.1.6).

Suppose that $x,y\in\langle w\rangle$ for some $w\in F$.
Then,
there exists $m,n\in\N$ such that $x=w^m$ and $y=w^n$.
Since $x\in\ker(\pi_1)$ and $y\in\ker(\pi_2)$,
it follows that $w^{mn}=e$,
i.e., $w$ is of finite order in a free product.
Therefore,
$w\in z\big\langle(g_{1;i},g_{2;i})\big\rangle z^{-1}$ for some $z\in F$ and $i\in I$ (see \cite{MagnusKarrasSolitar1966} Corollary 4.1.4).
Consequently,
$x\in\ker(\pi_1)\setminus\ker(\pi_2)$ and $y\in\ker(\pi_2)\setminus\ker(\pi_1)$ implies in all cases that
$x,y\in z\big\langle(g_{1;i},g_{2;i})\big\rangle z^{-1}$ for some $z\in F$ and $i\in I$.

Let $m,n\in\Z\setminus\{0\}$ be such that $x=z(g_{1;i},g_{2;i})^mz^{-1}$ and $y=z(g_{1;i},g_{2;i})^nz^{-1}$,
that is,
$(g_{1;i},g_{2;i})^m=z^{-1}xz$ and $(g_{1;i},g_{2;i})^n=z^{-1}yz$.
As $\ker(\pi_1)$ and $\ker(\pi_2)$ are normal subgroups,
$(g_{1;i},g_{2;i})^m\in\ker(\pi_1)\setminus\ker(\pi_2)$ and $(g_{1;i},g_{2;i})^n\in\ker(\pi_2)\setminus\ker(\pi_1)$.
According to Lemma \ref{Lemma: Direct Product Induces Relations},
this contradicts that $\D_\times$ is free,
hence $\ker(\pi_k)=\{e\}$ for at least one $k$,
which concludes the proof by Lemma \ref{Lemma: Kernel Criterion}.
\end{proof}


\section{Proof of Theorem \ref{Theorem: Main Theorem}}\label{Section: Main Result}

Given that the $*$-freeness of a collection of variables is unaffected by
scaling with nonzero constants,
we assume without loss of generality
that $\phi_k(a_{k;i}a_{k;i}^*)=\phi_k(a_{k;i}^*a_{k;i})=1$ for every $i\in I$ and $k\in \{1, \ldots , K\}$.
Since the $\phi_k$ are traces,
this amounts to multiplying $a_{k;i}$ by
$\big(\phi_k(a_{k;i}a_{k;i}^*)\big)^{-1/2}$ for all $k\in \{1, \ldots , K\}$ and $i\in I$.
Under these assumptions (where constant multiples of unitary variables become unitary variables),
claims (1), (2) and (3) in Theorem \ref{Theorem: Main Theorem} can be rephrased as follows.
\begin{enumerate}[$\qquad(1)'$]
\item For all $k\in \{1, \ldots , K\}$ except at most one,
every variable in $\a_k$ is unitary.
\item If there exists $k\in \{1, \ldots , K\}$ such that $\a_k$ contains a non-unitary variable,
then the TFC are satisfied with $\a_k$ as a dominating collection.
\item Suppose that $\a_1,\ldots,\a_k$ only contain unitary variables.
If there exists $k\in \{1, \ldots , K\}$,
$m\in\N$,
and $i\in I$ such that $\phi(D_i^m)\neq0$ and $a_{k;i}^m\neq\phi_k(a_{k;i}^m)$,
then $\D$ satisfies the TFC and $\a_k$ is a dominating collection.
\end{enumerate}

\begin{remark}
In contrast with the case of group algebras (see Lemma \ref{Lemma: Induction Step for Groups}),
it is not clear if
the proof of Theorem \ref{Theorem: Main Theorem} for an arbitrary $K\in\N$ can be reduced by induction to the case $K=2$,
especially since the $*$-freeness of a tensor product family $\diag(\a_{k(1)}\otimes\cdots\otimes\a_{k(t)})$
need not imply that one of the factor families $\a_{k(l)}$ is $*$-free
(see Example \ref{Example: Tensor Freeness Conditions Are Not Necessary}).
\end{remark}

\subsection{General Strategy of Proof}\label{Section: General Strategy of Proof}

The core idea behind the proof of Theorem \ref{Theorem: Main Theorem} is the following observation:
let $M\in\C\big\langle x_i,x_i^*:i\in I\rangle$ be a $*$-word in noncommuting indeterminates $x_i$ and $x_i^*$ ($i\in I$).
Then,
$M$ factors in the tensor product as
$$M(D_i:i\in I)
=M(a_{1;i}:i\in I)\otimes\cdots\otimes M(a_{K;i}:i\in I),$$
and thus,
by tensor (classical) independence,
the expected value also factors in the tensor product as
\begin{align}\label{Equation: Method of Proof 1}
\phi\big(M(D_i:i\in I)\big)=\phi_1\big(M(a_{1;i}:i\in I)\big)\cdots\phi_K\big(M(a_{K;i}:i\in I)\big).
\end{align}
As we have assumed that the collections $\a_1,\ldots,\a_K$ are $*$-free,
the right-hand side of (\ref{Equation: Method of Proof 1}) can be reduced
as an expression that only depends on the distributions of the variables in the collections $\a_k$
according to the rules of $*$-freeness.
If it is also assumed that $\D$ is $*$-free,
then the left-hand side of $(\ref{Equation: Method of Proof 1})$ can be reduced
as an expression that only depends on the distributions of the variables in the tensor product collection $\D$.
Consider for example the following computation,
which we will repeatedly use in this section.

\begin{example}\label{Example: Strategy}
Let $M,N\in\C\langle x,x^*\rangle$ be $*$-words and $i,j\in I$ be distinct indices.
A special case of
(\ref{Equation: Method of Proof 1}) is the following:
\begin{align}\label{Equation: Method of Proof 2}
\phi\big(M(D_i)N(D_j)M(D_i)^*N(D_j)^*\big)=\prod_{k=1}^K\phi_k\big(M(a_{k;i})N(a_{k;j})M(a_{k;i})^*N(a_{k;j})^*\big).
\end{align}
If $\a_1,\ldots,\a_K$ and $\D$ are $*$-free,
then applying (\ref{Equation: *-freeness Computation 1}) to the right-hand side and the left-hand side of the above equation yields
(we use $M_{k;i}:=M(a_{k;i})$ and $N_{k;j}:=N(a_{k;j})$ to alleviate notation)
\begin{multline}\label{Equation: Method of Proof 3}
\prod_{k=1}^K\big|\phi_k\big(M_{k;i}\big)\big|^2\phi_k\big(N_{k;j}N_{k;j}^*\big)
+\prod_{k=1}^K\phi_k\big(M_{k;i}M_{k;i}^*\big)\big|\phi_k\big(N_{k;j}\big)\big|^2
-\prod_{k=1}^K\big|\phi_k\big(M_{k;i}\big)\big|^2\big|\phi_k\big(N_{k;j}\big)\big|^2\\
=\prod_{k=1}^K\Big(\big|\phi_k\big(M_{k;i}\big)\big|^2\phi_k\big(N_{k;j}N_{k;j}^*\big)
+\phi_k\big(M_{k;i}M_{k;i}^*\big)\big|\phi_k\big(N_{k;j}\big)\big|^2
-\big|\phi_k\big(M_{k;i}\big)\big|^2\big|\phi_k\big(N_{k;j}\big)\big|^2\Big).
\end{multline}
\end{example}

The crucial observation we make is that,
in general (i.e., if $\D$ is not assumed to be $*$-free),
equations such as (\ref{Equation: Method of Proof 3}) need not hold.
Thus,
if we suppose that $\D$ is $*$-free,
then equation (\ref{Equation: Method of Proof 1}) with different choices of $*$-words
offers
in principle
infinitely many necessary conditions for the $*$-freeness of $\D$ under the assumption that $\a_1,\ldots,\a_K$ are $*$-free.
Equation (\ref{Equation: Method of Proof 3}) is an example of such a necessary condition.

\subsection{Theorem \ref{Theorem: Main Theorem} Claim $(1)'$}\label{Sub: At Most One Non-Unitary}

Let $b$ be a random variable in a $*$-probability space $(\BC,\psi)$ where $\psi$ is a faithful trace,
and assume that $\psi(bb^*)=\psi(b^*b)=1$.
To show that $b$ is unitary (i.e., $b^*b=bb^*=1=\phi(bb^*)$),
it suffices to prove that 
$\Var[bb^*]=\psi\big((bb^*)^2\big)-1=0,$
that is,
$\psi\big((bb^*)^2\big)=1$
(see Proposition \ref{Proposition: Variance with Faithful Functionals}).
Thus,
claim $(1)'$ of Theorem \ref{Theorem: Main Theorem} will be proved if it is shown that
for all $k\in \{1, \ldots , K\}$ except at most one,
$\phi_k\big((a_{k;i}a_{k;i}^*)^2\big)=1$ for all $i\in I$.

\begin{proof}[Proof of Theorem \ref{Theorem: Main Theorem} Claim $(1)'$]
If $\phi_k\big((a_{k;i}a_{k;i}^*)^2\big)=1$ for all $k\in \{1, \ldots , K\}$ and $i\in I$,
then the result trivially holds.
Thus,
assume without loss of generality that there exists $i\in I$ such that $\phi_1\big((a_{1;i}a_{1;i}^*)^2\big)\neq1$.
We prove that for every $k\in\{2,\ldots,K\}$ and $j\in I$,
one has $\phi_k\big((a_{k;j}a_{k;j}^*)^2\big)=1$.

Let $j\in I\setminus\{i\}$ be arbitrary.
According to equation
(\ref{Equation: Method of Proof 3})
with $M$ and $N$ defined as $M(x)=N(x)=xx^*$,
if $\D$ is $*$-free,
then (recall that $\phi_k(a_{k;i}a_{k;i}^*)=\phi_k(a_{k;j}a_{k;j}^*)=1$ for all $k\in \{1, \ldots , K\}$)
\begin{multline}\label{Equation: Main Theorem Step 1 eq2}
\prod_{k=1}^K\phi_k\big((a_{k;j}a_{k;j}^*)^2\big)+\prod_{k=1}^K\phi_k\big((a_{k;i}a_{k;i}^*)^2\big)-1\\
=\prod_{k=1}^K\Big(\phi_k\big((a_{k;j}a_{k;j}^*)^2\big)+\phi_k\big((a_{k;i}a_{k;i}^*)^2\big)-1\Big).
\end{multline}
Notice that,
according to equation (\ref{Equation: Norm of Element Condition}),
$1\leq\phi_k\big((a_{k;i}a_{k;i}^*)^2\big),\phi_k\big((a_{k;j}a_{k;j}^*)^2\big)$ for all $k\in \{1, \ldots , K\}$.
Whenever real numbers $1\leq x_1,\ldots,x_K,y_1,\ldots,y_K$ are such that
$$\prod_{k=1}^Kx_k+\prod_{k=1}^Ky_k-1=\prod_{k=1}^K(x_k+y_k-1),$$
it follows that $(x_k-1)(y_l-1)=0$ for every distinct $k,l\leq K$ (see Proposition \ref{Proposition: Polynomial Expansion 5}).
Applying this to equation (\ref{Equation: Main Theorem Step 1 eq2}) implies that for all $k\in\{2,\ldots,K\}$,
\begin{align}\label{Equation: Main Theorem Step 1 eq1}
\Big(\phi_k\big((a_{k;j}a_{k;j}^*)^2\big)-1\Big)\Big(\phi_1\big((a_{1;i}a_{1;i}^*)^2\big)-1\Big)=0.
\end{align}
As $\phi_1\big((a_{1;i}a_{1;i}^*)^2\big)\neq1$,
we conclude that $\phi_k\big((a_{k;j}a_{k;j}^*)^2\big)=1$
for all $k\in\{2,\ldots,K\}$ and $j\in I\setminus\{i\}$.

In order to complete the present proof,
it only remains to show that $\phi_k\big((a_{k;i}a_{k;i}^*)^2\big)=1$ for all $k\in\{2,\ldots,K\}$.
For this purpose,
fix $j\in I\setminus\{i\}$.
As shown in the previous paragraph,
$\phi_k\big((a_{k;j}a_{k;j}^*)^2\big)=1$
for all $k\in\{2,\ldots,K\}$.
We divide the remainder of this proof into the following cases:
\begin{enumerate}
\item $\phi_1\big((a_{1;j}a_{1;j}^*)^2\big)\neq1$ (i.e., $a_{1;j}$ is not unitary); and
\item $\phi_1\big((a_{1;j}a_{1;j}^*)^2\big)=1$ (i.e., $a_{1;j}$ is unitary).
\end{enumerate}
\noindent{\bf (1).}
Suppose that $\phi_1\big((a_{1;j}a_{1;j}^*)^2\big)\neq1$.
Then,
it follows from (\ref{Equation: Main Theorem Step 1 eq2}) that
$$\Big(\phi_1\big((a_{1;j}a_{1;j}^*)^2\big)-1\Big)\Big(\phi_k\big((a_{k;i}a_{k;i}^*)^2\big)-1\Big)=0$$
for all $k\in\{2,\ldots,K\}$,
which implies that $\phi_k\big((a_{k;i}a_{k;i}^*)^2\big)=1$,
as desired.

\noindent{\bf (2).}
Suppose that $\phi_1\big((a_{1;j}a_{1;j}^*)^2\big)=1$.
This implies in particular that $D_j$ is unitary,
hence (\ref{Equation: Main Theorem Step 1 eq2}) gives no useful information.
Instead, we apply equation (\ref{Equation: Method of Proof 3}) with $M(x)=xx^*$ and $N(x)=x$,
which yields
\begin{multline}\label{Equation: Main Theorem Step 1 eq5}
1+\prod_{k=1}^{K}\phi_k\big((a_{k;i}a_{k;i}^*)^2\big)|\phi_k(a_{k;j})|^2-\prod_{k=1}^{K}|\phi_k(a_{k;j})|^2\\
=\prod_{k=1}^{K}\Big(1+\phi_k\big((a_{k;i}a_{k;i}^*)^2\big)|\phi_k(a_{k;j})|^2-|\phi_k(a_{k;j})|^2\Big).
\end{multline}
Notice that
(\ref{Equation: Norm of Element Condition})
implies that for all $k\in \{1, \ldots , K\}$,
one has $1\leq\phi_k\big((a_{k;i}a_{k;i}^*)^2\big)$ and $0\leq |\phi_k(a_{k;j})|^2\leq 1$.
Given real numbers $1\leq x_1,\ldots,x_n$ and $0\leq t_1,\ldots,t_n\leq 1$ such that
$$1+\prod_{k=1}^Kx_kt_k-\prod_{k=1}^Kt_k=\prod_{k=1}^K\big(1+x_kt_k-t_k\big),$$
it follows
from Proposition \ref{Proposition: Polynomial Expansion 3} that for every $k\in \{1, \ldots , K\}$,
$$\left(t_k-\prod_{l=1}^Kt_l\right)(x_k-1)=0.$$
Applying this to (\ref{Equation: Main Theorem Step 1 eq5}) yields
\begin{align}\label{Equation: Main Theorem Step 1 eq4}
\Big(|\phi_k(a_{k;j})|^2-|\phi(D_j)|^2\Big)\Big(\phi_k\big((a_{k;i}a_{k;i}^*)^2\big)-1\Big)=0
\end{align}
for all $k\in \{1, \ldots , K\}$.
We divide the remainder of the proof of this case in two sub-cases:
\begin{enumerate}[\qquad(\text{2.}1)]
\item $\phi(D_j)\neq0$; and
\item $\phi(D_j)=0$.
\end{enumerate} 
\noindent{\bf (2.1).}
Suppose that $\phi(D_j)\neq0$.
According to (\ref{Equation: Main Theorem Step 1 eq4}),
one has
$$\Big(|\phi_1(a_{1;j})|^2-|\phi(D_j)|^2\Big)\Big(\phi_1\big((a_{1;i}a_{1;i}^*)^2\big)-1\Big)=0.$$
Since
$\phi_1\big((a_{1;i}a_{1;i}^*)^2\big)\neq1$,
this means that
$$|\phi(D_j)|^2=\prod_{k=1}^K|\phi_k(a_{k;j})|^2=|\phi_1(a_{1;j})|^2.$$
Since $0\leq |\phi_k(a_{k;j})|^2\leq 1$ for all $k\in \{1, \ldots , K\}$,
this implies that $|\phi_k(a_{k;j})|^2=1$ for all $k\in\{2,\ldots,K\}$.
Since it is assumed that $\D$ does not contain constant multiples of $1$, and $\phi(D_jD_j^*)=1$,
it cannot be the case that $|\phi(D_j)|=1$ (see Proposition \ref{Proposition: Unitary Equal to Unit}).
Therefore,
$1\neq|\phi_1(a_{1;j})|^2=|\phi(D_j)|^2.$
This then means that
$|\phi_k(a_{k;j})|^2\neq|\phi(D_j)|^2$ for all $k\in\{2,\ldots,K\}$,
hence $\phi_k\big((a_{k;i}a_{k;i}^*)^2\big)=1$ by (\ref{Equation: Main Theorem Step 1 eq4}).

\noindent{\bf (2.2).}
Suppose that $\phi(D_j)=0$.
Then,
(\ref{Equation: Main Theorem Step 1 eq4}) becomes
$$|\phi_k(a_{k;j})|^2\Big(\phi_k\big((a_{k;i}a_{k;i}^*)^2\big)-1\Big)=0.$$
Thus,
we see that for all $k\in \{1, \ldots , K\}$,
\begin{align}\label{Equation: Main Theorem Step 1 eq6}
\text{if }\phi_k(a_{k;j})\neq0\text{, then }\phi_k\big((a_{k;i}a_{k;i}^*)^2\big)=1.
\end{align}
Let us define $\U=\big\{k\in \{1, \ldots , K\}:\phi_k\big((a_{k;i}a_{k;i}^*)^2\big)\neq1\big\}$.
We have assumed that $1\in\U$ in the beginning of this proof,
and we want to show that $2,\ldots,K\not\in\U$.
According to equation (\ref{Equation: Method of Proof 1}),
we have that
\begin{multline}\label{Equation: Main Theorem Step 1 eq7}
\phi\Big(D_j(D_iD_i^*)D_j^*(D_iD_i^*)D_j(D_iD_i^*)D_j^*(D_iD_i^*)\Big)\\
=\prod_{k=1}^K\phi_k\Big(a_{k;j}(a_{k;i}a_{k;i}^*)a_{k;j}^*(a_{k;i}a_{k;i}^*)a_{k;j}(a_{k;i}a_{k;i}^*)a_{k;j}^*(a_{k;i}a_{k;i}^*)\Big).
\end{multline}
On the one hand,
the fact that $D_j$ is unitary and $\phi(D_j)=0$ implies by (\ref{Equation: *-freeness Computation 2}) that
\begin{align*}
\phi\Big(D_j(D_iD_i^*)D_j^*(D_iD_i^*)D_j(D_iD_i^*)D_j^*(D_iD_i^*)\Big)
&=2|\phi(D_iD_i^*)|^2\phi\big((D_iD_i^*)^2\big)-|\phi(D_iD_i^*)|^4\\
&=2\prod_{k\in\U}\phi_k\big((a_{k;i}a_{k;i}^*)^2\big)-1.
\end{align*}
On the other hand,
if $k\not\in\U$,
then (since $a_{k;i}$ is unitary)
$$\phi_k\Big(a_{k;j}(a_{k;i}a_{k;i}^*)a_{k;j}^*(a_{k;i}a_{k;i}^*)a_{k;j}(a_{k;i}a_{k;i}^*)a_{k;j}^*(a_{k;i}a_{k;i}^*)\Big)
=\phi_k(a_{k;j}a_{k;j}^*a_{k;j}a_{k;j}^*)=1,$$
and if $k\in\U$,
then the fact that $a_{k;j}$ is unitary and $\phi_k(a_{k;j})=0$ (according to (\ref{Equation: Main Theorem Step 1 eq6}))
implies by (\ref{Equation: *-freeness Computation 2}) that
$$\phi_k\Big(a_{k;j}(a_{k;i}a_{k;i}^*)a_{k;j}^*(a_{k;i}a_{k;i}^*)a_{k;j}(a_{k;i}a_{k;i}^*)a_{k;j}^*(a_{k;i}a_{k;i}^*)\Big)=
2\phi_k\big((a_{k;i}a_{k;i}^*)^2\big)-1.$$
Thus,
(\ref{Equation: Main Theorem Step 1 eq7}) yields that
$$2\prod_{k\in\U}\phi_k\big((a_{k;i}a_{k;i}^*)^2\big)-1=\prod_{k\in\U}\Big(2\phi_k\big((a_{k;i}a_{k;i}^*)^2\big)-1\Big).$$
According to Proposition \ref{Proposition: Polynomial Expansion 5},
this equation cannot hold if $\U$ contains more than one element.
It then follows that $k\not\in\U$ for all $k\in\{2,\ldots,K\}$,
which concludes the proof.
\end{proof}

\subsection{Theorem \ref{Theorem: Main Theorem} Claim $(2)'$}\label{Sub: Tensor Freeness Conditions}

Assume without loss of generality that $\a_1$ contains at least one variable that is not unitary,
and hence $\a_2,\ldots,\a_K$ only contains unitary variables by Theorem \ref{Theorem: Main Theorem} Claim $(1)'$.
We must prove that the TFC hold with $\a_1$ as a dominating collection.

Let $b$ be a random variable in a $*$-probability space $(\BC,\psi)$ where $\psi$ is a faithful trace,
and assume that $\psi(bb^*)=1$.
To show that $b$ is deterministic (i.e., $b=\psi(b)$),
it suffices to prove that
$0=\Var[b]=\psi(bb^*)-\psi(b)\psi(b^*)=1-|\psi(b)|^2,$
that is,
$|\psi(b)|^2=1$ (see Proposition \ref{Proposition: Unitary Equal to Unit}).
Thus,
if we assume that
$\a_1$ contains non-unitary variables and that $\a_2,\ldots,\a_K$ contain unitary variables only,
the TFC hold with $\a_1$ as a dominating collection if and only if for every $i\in I$ and $*$-word $M$,
one has
\begin{align}\label{Equation: Main Theorem Step 2 eq0}
\big|\phi\big(M(D_i)\big)\big|^2=\prod_{k=1}^K\big|\phi_k\big(M(a_{k;i})\big)\big|^2=\big|\phi_1\big(M(a_{1;i})\big)\big|^2.
\end{align}
In order to isolate $\a_1$ to obtain (\ref{Equation: Main Theorem Step 2 eq0}),
we make use of the fact that $\a_1$ is the only collection containing non-unitary variables,
as illustrated in the following example.

\begin{example}\label{Example: Isolate Non Unitary}
Let $M$ be an arbitrary $*$-word.
Equation (\ref{Equation: Method of Proof 2}) yields
$$\phi\big((D_iD_i^*)M(D_j)(D_iD_i^*)M(D_j)^*\big)=\prod_{k=1}^K\phi_k\big((a_{k;i}a_{k;i}^*)M(a_{k;j})(a_{k;i}a_{k;i}^*)M(a_{k;j})^*\big).$$
Looking at the right-hand side of the above equation,
we notice that,
for $k=2,\ldots,K$,
the fact that $a_{k;i}$ and $a_{k;j}$ are unitary implies that
$$\phi_k\big((a_{k;i}a_{k;i}^*)M(a_{k;j})(a_{k;i}a_{k;i}^*)M(a_{k;j})^*\big)=\phi_k\big(M(a_{k;j})M(a_{k;j})^*\big)=1.$$
Therefore,
\begin{align}\label{Equation: Main Theorem Step 2 eq0.5}
\phi\big((D_iD_i^*)M(D_j)(D_iD_i^*)M(D_j)^*\big)=\phi_1\big((a_{1;i}a_{1;i}^*)M(a_{1;j})(a_{1;i}a_{1;i}^*)M(a_{1;j})^*\big).
\end{align}
\end{example}

\begin{proof}[Proof of Theorem \ref{Theorem: Main Theorem} Claim $(2)'$]
Since $\a_1$ contains a non-unitary variable,
there exists $i\in I$ such that $\phi_1\big((a_{1;i}a_{1;i}^*)^2\big)\neq1$.
Let $j\in I\setminus\{i\}$ be arbitrary
and let $M$ be any $*$-word.
Applying (\ref{Equation: *-freeness Computation 1}) to both sides of equation (\ref{Equation: Main Theorem Step 2 eq0.5}) yields
\begin{multline*}
\phi_1\big(M(a_{1;j})M(a_{1;j})^*\big)+\phi_1\big((a_{1;i}a_{1;i}^*)^2\big)\prod_{k=1}^K\big|\phi_k\big(M(a_{k;j})\big)\big|^2-\prod_{k=1}^K\big|\phi_k\big(M(a_{k;j})\big)\big|^2\\
=\phi_1\big(M(a_{1;j})M(a_{1;j})^*\big)+\phi_1\big((a_{1;i}a_{1;i}^*)^2\big)\big|\phi_1\big(M(a_{1;j})\big)\big|^2
-\big|\phi_1\big(M(a_{1;j})\big)\big|^2,
\end{multline*}
which reduces to
$$\Big(\phi_1\big((a_{1;i}a_{1;i}^*)^2\big)-1\Big)\prod_{k=1}^K\big|\phi_k\big(M(a_{k;j})\big)\big|^2=
\Big(\phi_1\big((a_{1;i}a_{1;i}^*)^2\big)-1\Big)\big|\phi_1\big(M(a_{1;j})\big)\big|^2.$$
As $\phi_1\big((a_{1;i}a_{1;i}^*)^2\big)\neq1$,
we conclude that for every $j\in I\setminus\{i\}$ and $*$-word $M$,
one has
\begin{align}\label{Equation: Main Theorem Step 2 eq1}
\big|\phi\big(M(D_j)\big)\big|^2=\big|\phi_1\big(M(a_{1;j})\big)\big|^2.
\end{align}

In order to complete the present proof,
it only remains to show that
$\big|\phi\big(M(D_i)\big)\big|^2=\big|\phi_1\big(M(a_{1;i})\big)\big|^2$
for every $*$-word $M$.
For this purpose,
fix an arbitrary $j\in I\setminus\{i\}$.
We divide the remainder of this proof into the two following cases:
\begin{enumerate}
\item $\phi_1\big((a_{1;j}a_{1;j}^*)^2\big)\neq1$ (i.e., $a_{1;j}$ is not unitary); and
\item $\phi_1\big((a_{1;j}a_{1;j}^*)^2\big)=1$ (i.e., $a_{1;j}$ is unitary).
\end{enumerate}
\noindent{\bf (1).}
Suppose that $\phi_1\big((a_{1;j}a_{1;j}^*)^2\big)\neq1$.
This case follows by repeating the first paragraph of this proof with $i$ and $j$ interchanged.

\noindent{\bf (2).}
Suppose that $\phi_1\big((a_{1;j}a_{1;j}^*)^2\big)=1$.
We divide the remainder of the proof into the following two sub-cases,
which are exhaustive according to (\ref{Equation: Main Theorem Step 2 eq1}):
\begin{enumerate}[\qquad(\text{2.}1)]
\item $\phi_1(a_{1;j})=0$; and
\item $\phi_1(a_{1;j})\neq0$ and $|\phi_k(a_{k;j})|^2=1$ for all $k\in\{2,\ldots,K\}$.
\end{enumerate}
\noindent{\bf(2.1).}
Suppose that $\phi_1(a_{1;j})=0$.
According to (\ref{Equation: Method of Proof 1}),
\begin{multline*}
\phi\Big(D_jM(D_i)D_j^*(D_iD_i^*)D_jM(D_i)^*D_j^*(D_iD_i^*)\Big)\\
=\prod_{k=1}^K\phi_k\Big(a_{k;j}M(a_{k;i})a_{k;j}^*(a_{k;i}a_{k;i}^*)a_{k;j}M(a_{k;i})^*a_{k;j}^*(a_{k;i}a_{k;i}^*)\Big).
\end{multline*}
On the one hand,
since $D_j$ is unitary and $\phi(D_j)=0$,
(\ref{Equation: *-freeness Computation 2}) implies that (recall that $a_{k;i}$ is unitary for $k=2,\ldots,K$)
\begin{align*}
&\phi\Big(D_jM(D_i)D_j^*(D_iD_i^*)D_jM(D_i)^*D_j^*(D_iD_i^*)\Big)\\
=&\big|\phi\big(M(D_i)\big)\big|^2\phi\big((D_iD_i^*)^2\big)
+\phi\big(M(D_i)M(D_i)^*\big)\big|\phi(D_iD_i^*)\big|^2
-\big|\phi\big(M(D_i)\big)\big|^2\big|\phi(D_iD_i^*)\big|^2\\
=&\phi_1\big((a_{1;i}a_{1;i}^*)^2\big)\prod_{k=1}^K\big|\phi_k\big(M(a_{k;i})\big)\big|^2+\phi_1\big(M(a_{1;i})M(a_{1;i})^*\big)-\prod_{k=1}^K\big|\phi_k\big(M(a_{k;i})\big)\big|^2.
\end{align*}
On the other hand,
if $k\neq 1$,
then the fact that $a_{k;i}$ and $a_{k;j}$ are unitary imply that
$$\phi_k\Big(a_{k;j}M(a_{k;i})a_{k;j}^*(a_{k;i}a_{k;i}^*)a_{k;j}M(a_{k;i})^*a_{k;j}^*(a_{k;i}a_{k;i}^*)\Big)=1,$$
and if $k=1$,
then the fact that $a_{1;j}$ is unitary and $\phi_1(a_{1;j})=0$ implies by (\ref{Equation: *-freeness Computation 2}) that
\begin{multline*}
\phi_1\Big(a_{1;j}M(a_{1;i})a_{1;j}^*(a_{1;i}a_{1;i}^*)a_{1;j}M(a_{1;i})^*a_{1;j}^*(a_{1;i}a_{1;i}^*)\Big)\\
=\phi_1\big((a_{1;i}a_{1;i}^*)^2\big)\big|\phi_1\big(M(a_{1;i})\big)\big|^2+\phi_1\big(M(a_{1;i})M(a_{1;i})^*\big)-\big|\phi_1\big(M(a_{1;i})\big)\big|^2.
\end{multline*}
Therefore,
we conclude that
$$\Big(\phi_1\big((a_{1;i}a_{1;i}^*)^2\big)-1\Big)\prod_{k=1}^K\big|\phi_k\big(M(a_{k;i})\big)\big|^2=
\Big(\phi_1\big((a_{1;i}a_{1;i}^*)^2\big)-1\Big)\big|\phi_1\big(M(a_{1;i})\big)\big|^2,$$
which implies that $|\phi(D_i)|^2=\big|\phi_1\big(M(a_{1;i})\big)\big|^2$,
as desired.

\noindent{\bf(2.2).}
Suppose that $\phi_1(a_{1;j})\neq0$ and $|\phi_k(a_{k;j})|^2=1$ for all $k=2,\ldots,K$.
Then,
$$\prod_{k=2}^K\left(\big|\phi_k\big(M(a_{k;i})\big)\big|^2+|\phi_k(a_{k;j})|^2-\big|\phi_k\big(M(a_{k;i})\big)\big|^2|\phi_k(a_{k;j})|^2\right)=1,$$
and thus it follows from equation (\ref{Equation: Method of Proof 3}) with
$M$ arbitrary and $N(x)=x$ that
\begin{multline*}
\prod_{k=1}^K\big|\phi_k\big(M(a_{k;i})\big)\big|^2
+\phi_1\big(M(a_{1;i})M(a_{1;i})^*\big)\prod_{k=1}^K|\phi_k(a_{k;j})|^2
-|\phi_1(a_{1;j})|^2\prod_{k=1}^K\big|\phi_k\big(M(a_{k;i})\big)\big|^2\\
=\big|\phi_1\big(M(a_{1;i})\big)\big|^2
+\phi_1\big(M(a_{1;i})M(a_{1;i})^*\big)|\phi_1(a_{1;j})|^2
-\big|\phi_1\big(M(a_{1;i})\big)\big|^2|\phi_1(a_{1;j})|^2,
\end{multline*}
which reduces to
$$\left(1-|\phi_1(a_{1;j})|^2\right)\prod_{k=1}^K\big|\phi_k\big(M(a_{k;i})\big)\big|^2=
\left(1-|\phi_1(a_{1;j})|^2\right)\big|\phi_1\big(M(a_{1;i})\big)\big|^2.$$
Since $|\phi_1(a_{1;j})|^2\neq 1$
(the opposite would imply that $D_j$ is a constant multiple of the unit vector),
we conclude that the theorem holds in this sub-case.
\end{proof}

\subsection{Theorem \ref{Theorem: Main Theorem} Claim $(3)'$}\label{Sub: Tensor Freeness Conditions}

In this subsection,
we assume that all $a_{k;i}$ are unitary.

\begin{remark}
Any $*$-word $M$ evaluated in a unitary variable $u$
can be reduced to $M(u)=u^m$ for some integer $m$.
Thus,
to alleviate notation,
in this subsection,
we use integer powers of random variables instead of $*$-words evaluated in random variables.
\end{remark}

We begin with the following lemma.

\begin{lemma}\label{Lemma: Main Theorem Step 3-1}
If at least one of $\phi(D_i^m)$ and $\phi(D_j^n)$ is nonzero,
then for every
distinct $k,l\leq K$,
one has
\begin{align}\label{Equation: Main Theorem Step 3 eq1}
|\phi_k(a_{k;i}^m)|^2|\phi_l(a_{l;j}^n)|^2\big(1-|\phi_l(a_{l;i}^m)|^2\big)\big(1-|\phi_k(a_{k;j}^n)|^2\big)=0.
\end{align}
\end{lemma}
\begin{proof}
Since $a_{k;i}$ and $a_{k;j}$ are unitary for all $k$,
it follows from equation (\ref{Equation: Method of Proof 3}) with $M(x)=x^m$ and $N(x)=x^n$ that
\begin{multline}\label{Equation: Main Theorem Step 3 eq2}
\prod_{k=1}^K\big|\phi_k(a_{k;i}^m)\big|^2+\prod_{k=1}^K\big|\phi_k(a_{k;j}^n)\big|^2-\prod_{k=1}^K\big|\phi_k(a_{k;i}^m)\big|^2\big|\phi_k(a_{k;j}^n)\big|^2\\
=\prod_{k=1}^K\left(\big|\phi_k(a_{k;i}^m)\big|^2+\big|\phi_k(a_{k;j}^n)\big|^2-\big|\phi_k(a_{k;i}^m)\big|^2\big|\phi_k(a_{k;j}^n)\big|^2\right).
\end{multline}
Moreover, 
we notice that $\big|\phi_k(a_{k;i}^m)\big|^2,\big|\phi_k(a_{k;j}^n)\big|^2\leq 1$ (see (\ref{Equation: Norm of Element Condition})).
Whenever real numbers $0\leq x_1,\ldots,x_K,y_1,\ldots,y_K\leq 1$ are such that
$$\prod_{k=1}^Kx_k+\prod_{k=1}^Ky_k-\prod_{k=1}^Kx_ky_k=\prod_{k=1}^K(x_k+y_k-x_ky_k)$$
and $x_1,\ldots,x_K\neq0$ or $y_1,\ldots,y_K\neq 0$,
one has
$x_ky_l(1-x_l)(1-y_k)=0$ for all distinct $k,l\leq K$ (see Proposition \ref{Proposition: Polynomial Expansion 8}).
Applying this result to (\ref{Equation: Main Theorem Step 3 eq2}) concludes the proof of the lemma.
\end{proof}

Suppose without loss of generality that there exists $i\in I$ and $p\in\N$ such that $\phi(D_i^p)\neq0$ and $|\phi_1(a_{1;i}^p)|\neq1$
(i.e., $a_{1;i}^p\neq\phi_1(a_{1;i}^p)$).
We must prove that the TFC hold with $\a_1$ as a dominating collection.
We separate the proof in two propositions.

\begin{proposition}\label{Proposition: Unitaries Part 1}
For all $n\in\N$ and $j\in I$,
if $\phi(D_j^n)\neq0$,
then $|\phi_k(a_{k;j}^n)|=1$ whenever $k\in\{2,\ldots,K\}$.
\end{proposition}
\begin{proof}
Let $j\in I\setminus\{i\}$ and $n\in\N$ be arbitrary.
Since $\phi(D_i^p)\neq0$,
it follows from Lemma \ref{Lemma: Main Theorem Step 3-1} that for all $k\in\{2,\ldots,K\}$,
$$|\phi_k(a_{k;i}^p)|^2|\phi_1(a_{1;j}^n)|^2\big(1-|\phi_1(a_{1;i}^p)|^2\big)\big(1-|\phi_k(a_{k;j}^n)|^2\big)=0.$$
As $|\phi_1(a_{1;i}^p)|\neq1$ and $|\phi_k(a_{k;i}^p)|^2\neq0$,
we conclude that $|\phi_1(a_{1;j}^n)|^2\big(1-|\phi_k(a_{k;j}^n)|^2\big)=0$.
Therefore,
if $\phi(D_j^n)\neq0$,
then $|\phi_k(a_{k;j}^n)|=1$ for all $k\in\{2,\ldots,K\}$,
as desired.

We must now prove that for all $n\neq p$,
if $\phi(D_i^n)\neq0$,
then $|\phi_k(a_{k;i}^n)|=1$ for all $k\in\{2,\ldots,K\}$.
For this purpose,
let us fix $j\in I\setminus\{i\}$.
We divide the remainder of the proof into the two following cases:
\begin{enumerate}
\item $\phi(D_j)\neq0$; and
\item $\phi(D_j)=0$.
\end{enumerate}

\noindent{\bf (1).}
If $\phi(D_j)\neq0$,
then
$|\phi_k(a_{k;j})|=1$ for all $k\in\{2,\ldots,K\}$,
yet
$|\phi_1(a_{1;j})|\neq 1$
as the opposite implies that $D_j$ is a constant multiple of the unit vector.
The same application of Lemma \ref{Lemma: Main Theorem Step 3-1} as in the first paragraph of this proof
with $i$ and $j$ interchanged and $p=1$
implies that $|\phi_1(a_{1;i}^n)|^2\big(1-|\phi_k(a_{k;i}^n)|^2\big)=0$ for all $n\in\N$ and $k\in\{2,\ldots,K\}$,
as deisred.

\noindent{\bf (2).}
Suppose that $\phi(D_j)=0$.
Define the set
$\O=\{k\in \{1, \ldots , K\}:\phi_k(a_{k;j})=0\}$,
which is nonempty.
Suppose that $\phi(D_i^n)\neq0$,
and let $k_0\in\O$ be fixed.
If $k\not\in\O$,
then it follows from (\ref{Equation: Main Theorem Step 3 eq1}) that
$$|\phi_{k_0}(a_{k_0;i}^n)|^2|\phi_k(a_{k;j})|^2\big(1-|\phi_k(a_{k;i}^n)|^2\big)\big(1-|\phi_{k_0}(a_{k_0;j})|^2\big)=0,$$
hence $|\phi_k(a_{k;i}^n)|^2=1$.
Thus,
if we define $\U_n=\big\{k\in \{1, \ldots , K\}:|\phi_k(a_{k;i}^n)|^2\neq1\big\}$,
then $\U_n\subset\O$.
According to (\ref{Equation: Method of Proof 1})
$$\phi\Big(D_jD_i^nD_j^{-1}D_i^nD_jD_i^{-n}D_j^{-1}D_i^{-n}\Big)=
\prod_{k=1}^K\phi_k\Big(a_{k;j}a_{k;i}^na_{k;j}^{-1}a_{k;i}^na_{k;j}a_{k;i}^{-n}a_{k;j}^{-1}a_{k;i}^{-n}\Big).$$
On the one hand,
since $\phi(D_j)=0$,
(\ref{Equation: *-freeness Computation 2}) implies that 
$$\phi\Big(D_jD_i^nD_j^{-1}D_i^nD_jD_i^{-n}D_j^{-1}D_i^{-n}\Big)=2\prod_{k\in\U_n}|\phi_k(a_{k;i}^n)|^2-\prod_{k\in\U_n}|\phi_k(a_{k;i}^n)|^4.$$
On the other hand,
if $k\not\in\U_n$ (i.e., $a_{k;i}$ is deterministic),
then
$$\phi_k\Big(a_{k;j}a_{k;i}^na_{k;j}^{-1}a_{k;i}^na_{k;j}a_{k;i}^{-n}a_{k;j}^{-1}a_{k;i}^{-n}\Big)=1,$$
and if $k\in\U_n\subset\O$,
then it follows from (\ref{Equation: *-freeness Computation 2}) that
$$\phi_k\Big(a_{k;j}a_{k;i}^na_{k;j}^{-1}a_{k;i}^na_{k;j}a_{k;i}^{-n}a_{k;j}^{-1}a_{k;i}^{-n}\Big)=2|\phi_k(a_{k;i}^n)|^2-|\phi_k(a_{k;i}^n)|^4.$$
Therefore,
$$2\prod_{k\in\U_n}|\phi_k(a_{k;i}^n)|^2-\prod_{k\in\U_n}|\phi_k(a_{k;i}^n)|^4=\prod_{k\in\U_n}\big(2|\phi_k(a_{k;i}^n)|^2-|\phi_k(a_{k;i}^n)|^4\big).$$
According to Proposition \ref{Proposition: Polynomial Expansion 8},
the above equation cannot hold if $\U_n$ contains more than one element.
Thus,
$\U_n$ is a singleton, or empty.

To prove the lemma in this case,
it only remains to show that if $\phi(D_i^n)\neq0$ and $\U_n\neq\em$,
then $\U_n=\{1\}$.
We already know that $\U_p=\{1\}$ (by assumption),
but suppose by contradiction that there exists $n\neq p$ such that $\U_n=\{l\}$ for some $l\in\{2,\ldots,K\}$ .
According to equation (\ref{Equation: Method of Proof 1}),
$$\phi\Big(D_jD_i^pD_j^{-1}D_i^nD_jD_i^{-p}D_j^{-1}D_i^{-n}\Big)
=\prod_{k=1}^K\phi_k\Big(a_{k;j}a_{k;i}^pa_{k;j}^{-1}a_{k;i}^na_{k;j}a_{k;i}^{-p}a_{k;j}^{-1}a_{k;i}^{-n}\Big).$$
On the one hand,
since $\phi(D_j)=0$,
$|\phi(D_i^p)|=|\phi_{1}(a_{1;i}^p)|$,
and $|\phi(D_i^n)|=|\phi_{l}(a_{l;i}^n)|$,
equation (\ref{Equation: *-freeness Computation 2}) implies that
$$\phi\Big(D_jD_i^pD_j^{-1}D_i^nD_jD_i^{-p}D_j^{-1}D_i^{-n}\Big)=|\phi_{1}(a_{1;i}^p)|^2+|\phi_{l}(a_{l;i}^n)|^2-|\phi_{1}(a_{1;i}^p)|^2|\phi_{l}(a_{l;i}^n)|^2.$$
On the other hand,
$$\phi_k\Big(a_{k;j}a_{k;i}^pa_{k;j}^{-1}a_{k;i}^na_{k;j}a_{k;i}^{-p}a_{k;j}^{-1}a_{k;i}^{-n}\Big)=1$$
for all $k\in \{1, \ldots , K\}$,
since there is always one of $a_{k;i}^p$ or $a_{k;i}^n$ that is deterministic thanks to the assumption that $1\neq l$.
Therefore,
$|\phi_{1}(a_{1;i}^m)|^2+|\phi_{l}(a_{l;i}^n)|^2-|\phi_{1}(a_{1;i}^m)|^2|\phi_{l}(a_{l;i}^n)|^2=1,$
which reduces to
$\big(|\phi_{1}(a_{l_1;i}^m)|^2-1\big)\big(1-|\phi_{l}(a_{l;i}^n)|^2\big)=0.$
Given this contradiction,
we conclude that $\U_n=\{1\}$ for all $n\in\N$,
as desired.
\end{proof}

\begin{proposition}
For all $n\in\N$ and $j\in I$,
if $\phi(D_j^n)=0$,
then $\phi_1(a_{1;j}^n)=0$.
\end{proposition}
\begin{proof}
Let $j\in I\setminus\{i\}$ be arbitrary.
By using the same arguments as in the first paragraph of the proof of Proposition \ref{Proposition: Unitaries Part 1},
$\phi(D_i^p)\neq0$ and $|\phi_1(a_{1;i}^p)|\neq1$ imply that $$|\phi_1(a_{1;j}^n)|^2\big(1-|\phi_k(a_{k;j}^n)|^2\big)=0$$ for all $n\in\N$ and $k\in\{2,\ldots,K\}$.
Therefore,
if $\phi(D_j^n)=0$,
then $\phi_1(a_{1;j}^n)=0$.

It now only remains to prove that $\phi(D_i^n)=0$ implies that $\phi_1(a_{1;i}^n)=0$.
For this purpose,
fix $j\in I\setminus\{i\}$.
We divide the remainder of the proof in two cases,
which we know are exhaustive thanks to Proposition \ref{Proposition: Unitaries Part 1}:
\begin{enumerate}
\item there exists $q\in\N$ such that $\phi(D_j^q)\neq0$ and $|\phi_1(a_{1;j}^q)|^2\neq 1$; and
\item for every $n\in\N$,
$\phi(D_j^n)\neq0$ implies that $|\phi_k(a_{k;j}^n)|^2=1$ for all $k\in \{1, \ldots , K\}$.
\end{enumerate}

\noindent{\bf (1).}
Suppose that there exists $q\in\N$ such that $\phi(D_j^q)\neq0$ and $|\phi_1(a_{1;j}^q)|^2\neq 1$.
This case follows by repeating the first paragraph of this proof with $i$ and $j$ interchanged.

\noindent{\bf (2).}
Suppose that  for every $n\in\N$,
$\phi(D_j^n)\neq0$ implies that $|\phi_k(a_{k;j}^n)|^2=1$ for all $k\in \{1, \ldots , K\}$.
Notice that
it cannot be the case that $\phi(D_j)\neq0$,
as this would imply that $D_j$ is a constant multiple of $1$.
Thus, $\phi(D_j)=0$.

Suppose that $\phi(D_i^n)=0$.
According to equation (\ref{Equation: Method of Proof 1}),
$$\phi(D_jD_i^nD_j^{-1}D_i^pD_jD_i^{-n}D_j^{-1}D_i^{-p})=\prod_{k=1}^K\phi_k(a_{k,j}a_{k;i}^na_{k;j}^{-1}a_{k;i}^pa_{k;j}a_{k;i}^{-n}a_{k;j}^{-1}a_{k:i}^{-p}),$$
On the one hand,
since $\phi(D_j)=0$ and $|\phi_k(a_{k;i}^p)|^2=1$ for $k=2,\ldots,K$,
it follows from (\ref{Equation: *-freeness Computation 2}) that
\begin{multline*}
\phi(D_jD_i^nD_j^{-1}D_i^pD_jD_i^{-n}D_j^{-1}D_i^{-p})\\=|\phi(D_i^n)|^2+|\phi(D_i^p)|^2-|\phi(D_i^n)|^2|\phi(D_i^p)|^2=|\phi_1(a_{1;i}^p)|^2.
\end{multline*}
On the other hand,
for any $k=2,\ldots,K$,
the fact that $a_{k;i}^p$ is deterministic
implies that
$$\phi_k(a_{k,j}a_{k;i}^na_{k;j}^{-1}a_{k;i}^pa_{k;j}a_{k;i}^{-n}a_{k;j}^{-1}a_{k:i}^{-p})=\phi_k(a_{k,j}a_{k;i}^na_{k;j}^{-1}a_{k;j}a_{k;i}^{-n}a_{k;j}^{-1})=1,$$
and since $\phi(D_j)=0$ implies that $\phi_1(a_{1;j})=0$ (as shown in the first paragraph of the present proof),
it follows from (\ref{Equation: *-freeness Computation 2})
that
$$\phi_1(a_{1,j}a_{1;i}^na_{1;j}^{-1}a_{1;i}^pa_{1;j}a_{1;i}^{-n}a_{1;j}^{-1}a_{1:i}^{-p})=|\phi_1(a_{1;i}^n)|^2+|\phi_1(a_{1:i}^p)|^2-|\phi_1(a_{1;i}^n)|^2|\phi_1(a_{1;i}^p)|^2.$$
Therefore,
$$0=|\phi_1(a_{1;i}^n)|^2-|\phi_1(a_{1;i}^n)|^2|\phi_1(a_{1;i}^p)|^2=|\phi_1(a_{1;i}^n)|^2\big(1-|\phi_1(a_{1;i}^p)|^2\big),$$
from which we conclude that $\phi_1(a_{1;i}^n)=0$, as desired.
\end{proof}


\section{Discussion}\label{Section: Discussion}

The main contribution of this paper is to establish that,
in certain cases,
the TFC characterize the $*$-freeness of tensor products of the form
$\D=\diag(\a_1\otimes\cdots\otimes\a_K)$.
In light of the fact that the TFC do not characterize the freeness of $\D$ in general,
a first step towards better understanding the mechanisms that give rise to $*$-freeness in tensor products
could be to settle whether or not the TFC are necessary in cases other than those presented in
Proposition \ref{Proposition: Main Group} and Theorem \ref{Theorem: Main Theorem}.
For instance,
while Example \ref{Example: Tensor Freeness Conditions Are Not Necessary} shows
that the TFC need not hold for $\D$ to be $*$-free if none of the factor families $\a_k$ are $*$-free,
it is not clear if this is still true when some
but not all
of $\a_1,\ldots,\a_K$ are $*$-free.
The methods we use to prove Theorem \ref{Theorem: Main Theorem} (Section \ref{Section: General Strategy of Proof}) rely heavily
on the assumption that all of the families $\a_1,\ldots,\a_K$ are $*$-free,
hence it appears that a different approach is needed to solve this case.
However,
a more pressing example in the context of our present results is that of an apparent ``missing case" from Theorem \ref{Theorem: Main Theorem},
which makes our result fall short of a complete characterization of the $*$-freeness of $\D$ in the case where
$\a_k$ is $*$-free and $\phi_k$ is a faithful trace for all $k\in \{1, \ldots , K\}$.
We devote the remainder of this section to explaining what this missing case is,
why
it is plausible for it to be characterized by the TFC,
and why we believe that it cannot be settled with the methods used in the present paper.

\subsection{The Missing Case}

For the remainder of Section \ref{Section: Discussion},
we assume that the hypotheses of Theorem \ref{Theorem: Main Theorem} are met,
that is,
$\a_1,\ldots,\a_K$ are $*$-free,
$\phi_1,\ldots,\phi_K$ are faithful traces,
and $\D$ does not contain the zero vector or a constant multiple of the unit vector.

The first claim of Theorem \ref{Theorem: Main Theorem} is that if $\D$ is $*$-free,
then (assuming an appropriate renormalization,
see the first paragraph of Section \ref{Section: Main Result}):
(a) either all of $\a_1,\ldots,\a_K$ only contain unitary variables,
or 
(b)
exactly one collection $\a_k$ contains non-unitary variables.
The case where one family $\a_k$ contains non-unitary variables is fully accounted for:
the TFC hold with $\a_k$ as a dominating collection.
Thus,
the missing case is the one where $\a_1,\ldots,\a_k$ all contain unitary variables,
and the additional hypothesis in claim (3) of Theorem \ref{Theorem: Main Theorem} does not hold,
that is:

\hangindent=12pt
{\it The collections $\a_1,\ldots,\a_K$ only contain unitary variables,
and whenever $n\in\N$ and $i\in I$ are such that $\phi(D_i^n)\neq0$,
then $D_i^n$ is deterministic.}

If it also happens that whenever
$n\in\N$ and $i\in I$ are such that $\phi(D_i^n)=0$,
then $\phi_1(a_{1;i}^n)=\cdots =\phi_K(a_{K;i}^n)=0$,
then we are reduced to the case of group algebras with the canonical trace,
for which we have a full characterization in Proposition \ref{Proposition: Main Group}.
Thus,
solving the missing case in Theorem \ref{Theorem: Main Theorem} amounts to answering the following question.

\begin{question}\label{Question: Missing Case}
Suppose that the hypotheses of Theorem \ref{Theorem: Main Theorem} are met,
and that $\D$ is $*$-free.
If for every $n\in\N$ and $i\in I$,
$\phi(D_i^n)\neq0$,
implies that $D_i^n$ is deterministic;
and there exists $p\in\N$, $i\in I$ and $k\in \{1, \ldots , K\}$ such that
$\phi(D_i^p)=0$ and $\phi_k(a_{k;i}^p)\neq0$,
then does the TFC hold?
\end{question}

\subsection{The case $K=2$}

The most compelling evidence we have that the TFC might provide a characterization for the missing case is that
we can answer Question \ref{Question: Missing Case} in the affirmative when $K=2$.

\begin{lemma}\label{Lemma: Remaining Cases Lemma}
Suppose that there exists
$p\in\N$, $i\in I$ and $k\in \{1, \ldots , K\}$ such that $\phi(D_i^p)=0$ and $\phi(a_{k;i}^p)\neq0$.
For every $j\in I\setminus\{i\}$ and $n\in\N$,
if $\phi(D_j^n)=0$,
then there exists $l\neq k$ such that $\phi_l(a_{l;i}^p)=\phi(a_{l;j}^n)=0$.
\end{lemma}

\begin{proof}
Since $\D$ is $*$-free,
if $\phi(D_j^n)=0$,
then
$\phi(D_i^p)=0$ implies that
$\phi(D_i^pD_j^nD_i^{-p}D_j^{-n})=0.$
Thus,
it follows from (\ref{Equation: Method of Proof 3}) that
$$0=\prod_{k=1}^K\big(|\phi_k(a_{k;i}^p)|+|\phi_k(a_{k;j}^n)|-|\phi_k(a_{k;i}^p)||\phi_k(a_{k;j}^n)|\big).$$
Given that $0\leq x,y\leq 1$ are such that $x+y-xy=0$ if and only if $x=y=0$,
we conclude that there must be an $l\neq k$ such that $\phi_l(a_{l;i}^p)=\phi(a_{l;j}^n)=0$.
\end{proof}

\begin{proof}[Proof of Question \ref{Question: Missing Case} for $K=2$]
Suppose without loss of generality that there exists $i\in I$ and $p\in\N$ such that
$\phi(a_{1;i}^p\otimes a_{2;i}^p)=0$ and $\phi_2(a_{2;i}^p)\neq0$.
Then, Lemma \ref{Lemma: Remaining Cases Lemma} implies that
for all $j\in I\setminus\{i\}$ and $n\in\N$,
if $\phi(a_{1;j}^n\otimes a_{2;j}^n)=0$,
then $\phi_1(a_{1;j}^n)=0$.
 
It only remains to prove that 
for all $n\in\N$,
if $\phi(D_i^n)=0$,
then $\phi_1(a_{1;i}^n)=0.$
For this purpose,
let us fix $j\in I\setminus\{i\}$.
We separate the proof in the following cases:
\begin{enumerate}
\item there exists $q\in\N$ such that $\phi(a_{1;j}^q\otimes a_{2;j}^q)=0$ and $\phi_2(a_{2;j}^q)\neq0$; and
\item for all $n\in\N$, $\phi(a_{1;i}^n\otimes a_{2;i}^n)=0$ implies that $\phi_1(a_{1;j}^n)=\phi_2(a_{2;j}^n)=0$.
\end{enumerate}

\noindent{\bf (1).}
This case follows from Lemma \ref{Lemma: Remaining Cases Lemma} by using the same arguments as in the first paragraph of this proof.

\noindent{\bf (2).}
In this case we consider the equality
$$\phi(D_i^nD_j^qD_i^pD_j^{-q}D_i^{-n}D_j^qD_i^{-p}D_j^{-q})=\prod_{k=1}^2\phi_k(a_{k;i}^na_{k;j}^qa_{k;i}^pa_{k;j}^{-q}a_{k;i}^{-n}a_{k;j}^qa_{k;i}^{-p}a_{k;j}^{-q}).$$
After an application of 
equation (\ref{Equation: *-freeness Computation 2}),
the result follows using the same argument as Lemma \ref{Lemma: Remaining Cases Lemma}.
\end{proof}

\subsection{The case $K\geq 3$}

We now argue that the argument used for $K=2$ above cannot work for the case $K\geq 3$ in general.
To illustrate this,
consider the following example:
Let $K\geq3$ be fixed.
Let $D_i=a_{1;i}\otimes \cdots\otimes a_{K;i}$ be unitary, of infinite order,
and suppose that the distributions of the $a_{k;i}$ are given by Table \ref{Table} below.
\begin{table}[htdp]
\caption{Counterexample}
$$\begin{array}{|c|cccc|}
\hline
n&\phi_1(a_{1;i}^n)&\phi_2(a_{2;i}^n)&\cdots&\phi_K(a_{K;i}^n)\\
\hline
\hline
1&\al&0&\cdots&0\\
2&0&\al&\cdots&0\\
\vdots&\vdots&\vdots&\ddots&\vdots\\
K&0&0&\cdots&\al\\
K+1&0&0&\cdots&0\\
\vdots&\vdots&\vdots&\vdots&\vdots\\
\hline
\end{array}$$
\label{Table}
\end{table}
If we ensure that $\phi_k(a_{k;i}^{-n})=\bar{\phi_k(a_{k;i}^n)}$ for all $n$ and we choose $0<|\al|$ small enough,
the $\phi_k$ are positive faithful traces when restricted to the $*$-algebras generated by the $a_{k;i}$.
Let $D_j=a_{1;j}\otimes \cdots\otimes a_{K;j}$ be such that $a_{1;j},\ldots,a_{K;j}$ are Haar unitary,
and such that $(a_{k;i},a_{k;j})$ is $*$-free for all $k\leq K$.
It is easy to see that this example is part of the missing case alluded to in Question \ref{Question: Missing Case}.

\begin{question}
Can we prove that $(D_i,D_j)$ is not $*$-free (as the opposite would violate the conjecture that the TFC are necessary in the missing case)?
\end{question}

The method we use to prove Theorem \ref{Theorem: Main Theorem} 
essentially relies on the following procedure:
\begin{enumerate}
\item Fix a $*$-word $M\in\C\langle x_i,x_i^*:i\in I\rangle$ such that,
if the $x_i$ are free,
we can easily compute $\phi\big(M(x_i:i\in I)\big)$ explicitly in terms of the distributions of the $x_i$
(for example,
through formulas such as (\ref{Equation: *-freeness Computation 1}) and (\ref{Equation: *-freeness Computation 2}));
\item use the identity
$$\phi\big(M(D_i:i\in I)\big)=\phi_1\big(M(a_{1;i}:i\in I)\big)\cdots\phi_K\big(M(a_{K;i}:i\in I)\big);$$
to obtain contradictions whenever the TFC are not satisfied.
\end{enumerate}
The problem with the present example is that there is no finite $*$-word that can rule out the $*$-freeness of $(D_i,D_j)$ for all values of $K$,
which suggests that the TFC do not characterize the missing case,
or that another method is needed.
We now substantiate this claim.

Let $M(x_i,x_j)$ be a $*$-word,
which we can write as
$$M(x_i,x_j)=x_i^{n_1}x_j^{n_2}\cdots x_i^{n_{2t-1}}x_j^{n_{2t}},$$
or
$$M(x_i,x_j)=x_i^{n_1}x_j^{n_2}\cdots x_i^{n_{2t-1}}$$
for some $n_1,\ldots,n_{2t}\in\N$.
Then,
to disprove the $*$-freeness of $(D_i,D_j)$ using this $*$-word,
we must prove that
$$\phi\big(M(D_i,D_j)\big)=\prod_{k=1}^K\phi\big(M(a_{k;i},a_{k;j})\big)\neq0.$$
According to the moment-cumulant formula
(see \cite{NicaSpeicher2006} Lectures 11 and 14 for a definition of free cumulants,
the moment-cumulant formula,
and the notation used in the remainder of this section),
one has
\begin{multline*}
\prod_{k=1}^K\phi\big(M(a_{k;i},a_{k;j})\big)=\prod_{k=1}^K\phi_k(a_{k;i}^{n_1}a_{k;j}^{n_2}\cdots a_{k;i}^{n_{2t-1}}a_{k;j}^{n_{2t}})\\
=\prod_{k=1}^K\sum_{\pi\in NC(2t)}\kappa_{\pi}[a_{k;i}^{n_1},a_{k;j}^{n_2},\cdots,a_{k;i}^{n_{2t-1}},a_{k;j}^{n_{2t}}].
\end{multline*}
If the partition $\pi$ has a block that does not contain only even or odd integers in $\{1,\ldots,2t\}$,
then the vanishing of mixed cumulants in $*$-free variables and the $*$-freeness of $(a_{k;i},a_{k;j})$ implies that
$\kappa_{\pi}[a_{k;i}^{n_1},a_{k;j}^{n_2},\cdots,a_{k;i}^{n_{2t-1}},a_{k;j}^{n_{2t}}]=0$.
Thus,
if we let $NC_{e,o}(2t)$ be the set of noncrossing partitions $\pi\in NC(2t)$ such that every block of $\pi$ contains even or odd elements only,
one has
\begin{align*}
\phi\big(M(D_i,D_j)\big)&=\prod_{k=1}^K\sum_{\pi\in NC_{e,o}(2t)}\kappa_{\pi}[a_{k;i}^{n_1},a_{k;j}^{n_2},\cdots,a_{k;i}^{n_{2t-1}},a_{k;j}^{n_{2t}}].
\end{align*}

If a noncrossing partition $\pi$ only contains even or odd elements,
then this necessarily means that $\pi$ contains at least one singleton.
If one of the singletons that $\pi$ contains is even,
then the fact that the $a_{k;j}$ are Haar unitary implies that
$$\kappa_{\pi}[a_{k;i}^{n_1},a_{k;j}^{n_2},\cdots,a_{k;i}^{n_{2t-1}},a_{k;j}^{n_{2t}}]=0.$$
Thus,
if we let $N_{o}=\{\pi\in NC_{e,o}(2t):\pi\text{ only has odd singletons}\}$,
one has
\begin{align*}
\phi\big(M(D_i,D_j)\big)&=\prod_{k=1}^K\sum_{\pi\in N_{o}}\kappa_{\pi}[a_{k;i}^{n_1},a_{k;j}^{n_2},\cdots,a_{k;i}^{n_{2t-1}},a_{k;j}^{n_{2t}}].
\end{align*}

Let $k\in \{1, \ldots , K\}$ be fixed.
If $\pi$ has an odd singleton $\{p\}$ such that $n_p\neq \pm k$,
then $\kappa_1[a_{k;i}^{n_p}]=\phi_k(a_{k;i}^{n_p})=0$ (see Table \ref{Table}) implies that 
$$\kappa_{\pi}[a_{k;i}^{n_1},a_{k;j}^{n_2},\cdots,a_{k;i}^{n_{2t-1}},a_{k;j}^{n_{2t}}]=0.$$
Thus,
if we let $N_o^{(k)}=\{\pi\in N_o:\text{if }\{p\}\in \pi,\text{ then }n_p=\pm k)\},$
one has
\begin{align*}
\phi\big(M(D_i,D_j)\big)&=\prod_{k=1}^K\sum_{\pi\in N_{o}^{(k)}}\kappa_{\pi}[a_{k;i}^{n_1},a_{k;j}^{n_2},\cdots,a_{k;i}^{n_{2t-1}},a_{k;j}^{n_{2t}}].
\end{align*}

We therefore conclude that the only way for
$\phi\big(M(D_i,D_j)\big)=0$
to possibly fail is if none of $N_{o}^{(k)}$ ($k\in \{1, \ldots , K\}$) are the empty set.
However,
notice that if
$N_{o}^{(k)}\neq\em$,
then there exists at least one odd $l$ such that $n_l=\pm k$,
and thus
if $N_{o}^{(k)}\neq\em$ for all $k$,
then $t\geq K$ (recall that $t$ is related to the size of the $*$-word $M(x_i,x_j)=x_i^{n_1}x_j^{n_2}\cdots x_i^{n_{2t-1}}x_j^{n_{2t}}$).
Therefore,
to get a contradiction for the $*$-freeness of $(D_i,D_j)$ (or even more general examples) using the methods in this paper for all $K$,
we would need to compute moments in $*$-words whose size grows to infinity with $K$,
which seems not to be feasible. 


\appendix
\section{Technical Results}\label{Appendix}

\begin{lemma}\label{Lemma: Polynomial Expansion 1}
Let $n\in\N$ be fixed,
and let $\al,x_1,\ldots,x_K$ be arbitrary complex
numbers.
Then,
\begin{align}\label{Equation: Polynomial Expansion 1}
\al\prod_{k=1}^K(x_k-1)=\al\prod_{k=1}^Kx_k-\al-\sum_{s=1}^{K-1}\left(\sum_{1\leq k(1)<\ldots<k(s)\leq K}\al\left(\prod_{l\leq s}(x_{k(l)}-1)\right)\right).
\end{align}
\end{lemma}

\begin{proof}
We proceed by induction.
For $K=1$ the result is trivial.
Suppose that equation (\ref{Equation: Polynomial Expansion 1}) holds for a fixed integer $K\in\N$.
Then
\begin{multline*}
\al\prod_{k=1}^{K+1}(x_k-1)
=\al\prod_{k=1}^K(x_k-1)(x_{K+1}-1)\\
=\left(\al\prod_{k=1}^Kx_k-\al\right)(x_{K+1}-1)
-\sum_{s=1}^{K-1}\left(\sum_{1\leq k(1)<\ldots<k(s)\leq K}\al\left(\prod_{l\leq s}(x_{k(l)}-1)\right)(x_{K+1}-1)\right).
\end{multline*}
On the one hand,
\begin{multline*}
\sum_{s=1}^{K-1}\left(\sum_{1\leq k(1)<\ldots<k(s)\leq K}\al\left(\prod_{l\leq s}(x_{k(l)}-1)(x_{K+1}-1)\right)\right)\\
=\sum_{s=2}^{K}\left(\sum_{1\leq k(1)<\ldots<k(s)=K+1}\al\left(\prod_{l\leq s}(x_{k(l)}-1)\right)\right).
\end{multline*}
On the other hand,
\begin{align*}
\left(\al\prod_{k=1}^Kx_k-\al\right)(x_{K+1}-1)&=\left(\al\prod_{k=1}^{K+1}x_k-\al\right)-\al(x_{K+1}-1)-\left(\al\prod_{k=1}^Kx_k-\al\right).
\end{align*}
Thus,
since
\begin{multline*}
\sum_{s=2}^{K}\left(\sum_{1\leq k(1)<\ldots<k(s)=K+1}\al\left(\prod_{l\leq s}(x_{k(l)}-1)\right)\right)+\al(x_{K+1}-1)\\
=\sum_{s=1}^{K}\left(\sum_{1\leq k(1)<\ldots<k(s)=K+1}\al\left(\prod_{l\leq s}(x_{k(l)}-1)\right)\right),
\end{multline*}
it follows that
\begin{multline}\label{Equation: Polynomial Expansion 1.1}
\al\prod_{k=1}^{K+1}(x_k-1)
=\left(\al\prod_{k=1}^{K+1}x_k-\al\right)-\left(\al\prod_{k=1}^Kx_k-\al\right)\\
-\sum_{s=1}^{K}\left(\sum_{1\leq k(1)<\ldots<k(s)=K+1}\al\left(\prod_{l\leq s}(x_{k(l)}-1)\right)\right).
\end{multline}
Given that equation (\ref{Equation: Polynomial Expansion 1}) holds for $K$,
we know that
\begin{multline*}
-\left(\al\prod_{k=1}^Kx_k-\al\right)=-\al\prod_{k=1}^K(x_k-1)-\sum_{s=1}^{K-1}\left(\sum_{1\leq k(1)<\ldots<k(s)\leq K}\al\left(\prod_{l\leq s}(x_{k(l)}-1)\right)\right)\\
=-\sum_{s=1}^{K}\left(\sum_{1\leq k(1)<\ldots<k(s)\leq K}\al\left(\prod_{l\leq s}(x_{k(l)}-1)\right)\right).
\end{multline*}
By combining the above with equation (\ref{Equation: Polynomial Expansion 1.1}),
we conclude that the result holds for $K+1$ as well.
\end{proof}

\begin{lemma}\label{Lemma: Polynomial Expansion 2}
Let $K\in\N$ be fixed,
and let $x_1,\ldots,x_K$ and $t_1,\ldots,t_K$ be arbitrary complex numbers.
Let us denote $\al=\prod_{k=1}^Kt_k$.
Then,
\begin{multline}\label{Equation: Polynomial Expansion 2}
\prod_{k=1}^K(t_kx_k+1-t_k)\\
=\left(\al\prod_{k=1}^Kx_k+1-\al\right)
+\sum_{s=1}^{K-1}\left(\sum_{1\leq k(1)<\ldots<k(s)\leq K}\left(\prod_{l\leq s}t_{k(l)}-\al\right)\left(\prod_{l\leq s}(x_{k(l)}-1)\right)\right).
\end{multline}
\end{lemma}

\begin{proof}
First,
notice that
\begin{multline*}
\prod_{k=1}^K(t_kx_k+1-t_k)
=\prod_{k=1}^K\big(t_k(x_k-1)+1\big)\\
=1+\sum_{s=1}^{K-1}\left(\sum_{1\leq k(1)<\cdots<k(s)\leq K}\left(\prod_{l\leq s}t_{k(l)}(x_{k(l)}-1)\right)\right)+\al\prod_{k=1}^K(x_k-1).
\end{multline*}
The result then follows from Lemma \ref{Lemma: Polynomial Expansion 1}.
\end{proof}

\begin{proposition}\label{Proposition: Polynomial Expansion 3}
Let $K\in\N$ be fixed,
and let $1\leq x_1,\ldots,x_K$ and $0\leq t_1,\ldots,t_K\leq 1$ be real numbers.
Let us denote $\al=\prod_{k=1}^Kt_k$.
If 
$$\left(\al\prod_{k=1}^Kx_k+1-\al\right)=\prod_{k=1}^K(t_kx_k+1-t_k)$$
then,
$$\left(\prod_{l\leq s}t_{k(l)}-\al\right)\left(\prod_{l\leq s}(x_{k(l)}-1)\right)=0$$
for all $1\leq s\leq K-1$ and $1\leq k(1)\leq\cdots\leq k(s)\leq K$.
\end{proposition}

\begin{proof}
This is a direct consequence of Lemma \ref{Lemma: Polynomial Expansion 2}.
\end{proof}

\begin{lemma}\label{Lemma: Polynomial Expansion 4}
Let $K\in\N$ and $x_1,\ldots,x_K,y_1,\ldots,y_K\geq1$.
Then,
\begin{align}\label{Equation: Polynomial Expansion 4}
\prod_{k=1}^Kx_k+\prod_{k=1}^Ky_k-1\leq \prod_{k=1}^K(x_k+y_k-1).
\end{align}
\end{lemma}
\begin{proof}
We proceed by induction.
For $K=1$,
the result is trivial.
Thus,
let $K\geq2$,
and suppose that the result holds for $1,2,\ldots,K-1$.
We consider the two cases regarding the parity of the integer $K$.

\noindent{\bf(1).}
Suppose that $K$ is even.
For every $k<K$,
notice that
\begin{align*}
&(x_k+y_k-1)(x_{k+1}+y_{k+1}-1)\\
&=x_kx_{k+1}+x_ky_{k+1}-x_k+y_kx_{k+1}+y_ky_{k+1}-y_k-x_{k+1}-y_{k+1}+1\\
&=(x_ky_{k+1}-x_k-y_{k+1}+1)+(y_kx_{k+1}-x_{k+1}-y_k+1)+(x_kx_{k+1}+y_ky_{k+1}-1)\\
&=(x_k-1)(y_{k+1}-1)+(y_k-1)(x_{k+1}-1)+(x_kx_{k+1}+y_ky_{k+1}-1).
\end{align*}
Therefore,
since $K$ is even,
we can write
\begin{multline}\label{Equation: S 1}
\prod_{k=1}^{K}(x_k+y_k-1)\\
=\prod_{k\leq K\text{ odd}}\big((x_k-1)(y_{k+1}-1)+(y_k-1)(x_{k+1}-1)+(x_kx_{k+1}+y_ky_{k+1}-1)\big).
\end{multline}
For every odd $k<K$,
let $z_{k,k+1}^{(1)}=(x_k-1)(y_{k+1}-1)$,
$z_{k,k+1}^{(2)}=(y_k-1)(x_{k+1}-1)$, and
$z_{k,k+1}^{(3)}=(x_kx_{k+1}+y_ky_{k+1}-1)$,
and let
\begin{align}\label{Equation: S 2}
S=\sum z_{1,2}^{(i_1)}z_{3,4}^{(i_3)}\cdots z_{K-1,K}^{(i_{K-1})},
\end{align}
where the sum $S$ is taken over all collections $i_1,i_3,\ldots,i_{K-1}\in\{1,2,3\}$,
except for the collection $i_1=i_3=\cdots=i_{K-1}=3$.
As $x_1,\ldots,x_K,y_1,\ldots,y_K\geq1$,
it follows that $z_{k,k+1}^{(1)},z_{k,k+1}^{(2)},z_{k,k+1}^{(3)}\geq0$ for every odd $k$.
Therefore,
$S\geq0$.
Furthermore,
given that (\ref{Equation: Polynomial Expansion 4}) holds for $1,\ldots,K-1$
(in particular $K/2$),
and that $(x_kx_{k+1}),(y_ky_{k+1})\geq1,$ for all $k$,
\begin{multline*}
z_{1,2}^{(3)}z_{3,4}^{(3)}\cdots z_{K-1,K}^{(3)}=\prod_{k\leq K\text{ odd}}(x_kx_{k+1}+y_ky_{k+1}-1)\\
\geq\prod_{k\leq K\text{ odd}}x_kx_{k+1}+\prod_{k\leq K\text{ odd}}y_ky_{k+1}-1=\prod_{k=1}^{K}x_k+\prod_{k=1}^{K}y_k-1.
\end{multline*}
Therefore,
\begin{align*}
\prod_{k=1}^{K}(x_k+y_k-1)=S+z_{1,2}^{(3)}z_{3,4}^{(3)}\cdots z_{K-1,K}^{(3)}\geq0+\prod_{k=1}^{K}x_k+\prod_{k=1}^{K}y_k-1,
\end{align*}
which implies that the result holds for $K$.

\noindent{\bf(2).}
Suppose that $K$ is odd,
i.e.,
$K-1$ is even.
Define the quantities $z_{k,k+1}^{(1)}$, $z_{k,k+1}^{(2)}$ and $z_{k,k+1}^{(3)}$
for odd $k\in \{1, \ldots , K-1\}$ as in the previous case.
Since $K-1$ is even,
we can write
$$\prod_{k=1}^{K-1}(x_k+y_k-1)=\prod_{k\leq K-1\text{ odd}}\big(z_{k,k+1}^{(1)}+z_{k,k+1}^{(2)}+z_{k,k+1}^{(3)}\big).$$
Define again the sum
$S=\sum z_{1,2}^{(i_1)}z_{3,4}^{(i_3)}\cdots z_{K-2,K-1}^{(i_{K-2})}\geq0$
as ranging over all collections $i_1,\ldots,i_{K-2}\in\{1,2,3\}$
except $i_1=\cdots=i_{K-2}=3$.
Then,
$$\prod_{k=1}^{K}(x_k+y_k-1)=\big(S+z_{1,2}^{(3)}z_{3,4}^{(3)}\cdots z_{K-2,K-1}^{(3)}\big)(x_K+y_K-1).$$
Given that (\ref{Equation: Polynomial Expansion 4}) holds for $(K-1)/2$,
it follows that
\begin{multline*}
z_{1,2}^{(3)}z_{3,4}^{(3)}\cdots z_{K-2,K-1}^{(3)}\\
\geq\prod_{k\leq K-1\text{ odd}}x_kx_{k+1}+\prod_{k\leq K-1\text{ odd}}y_ky_{k+1}-1=\prod_{k=1}^{K-1}x_k+\prod_{k=1}^{K-1}y_k-1.
\end{multline*}
Therefore,
\begin{multline*}
\prod_{k=1}^{K}(x_k+y_k-1)
\geq\left(S+\prod_{k=1}^{K-1}x_k+\prod_{k=1}^{K-1}y_k-1\right)(x_K+y_K-1)\\
=S(x_K+y_K-1)+\left(\prod_{k=1}^{K-1}x_k+\prod_{k=1}^{K-1}y_k-1\right)(x_K+y_K-1).
\end{multline*}
Given that (\ref{Equation: Polynomial Expansion 4}) also holds for $2$,
then
\begin{multline*}
\left(\prod_{k=1}^{K-1}x_k+\prod_{k=1}^{K-1}y_k-1\right)(x_K+y_K-1)\\\geq\left(\prod_{k=1}^{K-1}x_k\right)x_K+\left(\prod_{k=1}^{K-1}y_k\right)y_K-1
=\prod_{k=1}^{K}x_k+\prod_{k=1}^{K}y_k-1,
\end{multline*}
and thus,
since $S(x_K+y_K-1)\geq0$,
it follows that
$$\prod_{k=1}^{K}(x_k+y_k-1)\geq0+\prod_{k=1}^{K}x_k+\prod_{k=1}^{K}y_k-1,$$
as desired.
\end{proof}

\begin{proposition}\label{Proposition: Polynomial Expansion 5}
Let $K\in\N$ and $x_1,\ldots,x_K,y_1,\ldots,y_K\geq1$.
If
$$\prod_{k=1}^Kx_k+\prod_{k=1}^Ky_k-1=\prod_{k=1}^K(x_k+y_k-1),$$
then $(x_k-1)(y_l-1)=0$ for every distinct $k,l\leq K$.
\end{proposition}

\begin{proof}
Let $S$ be defined as in (\ref{Equation: S 2}).
The hypothesis of the present proposition implies that $S=0$ if $K$ is even,
or $S(x_K+y_K-1)=0$ if $K$ is odd.
From this we directly infer that $(x_k-1)(y_{k+1}-1)=0$ for every odd $k\leq K-1$.
By rearranging the index set in (\ref{Equation: S 1}) as needed,
the conclusion follows.
\end{proof}

\begin{lemma}\label{Lemma: Polynomial Expansion 6}
Let $K\in\N$ and $0\leq x_1,\ldots,x_K,y_1,\ldots,y_K\leq1$.
Then,
$$\prod_{k=1}^Kx_k+\prod_{k=1}^Ky_k-\prod_{k=1}^Kx_ky_k\leq\prod_{k=1}^K(x_k+y_k-x_ky_k).$$
\end{lemma}
\begin{proof}
For every $k<K$,
\begin{multline*}
(x_k+y_k-x_ky_k)(x_{k+1}+y_{k+1}-x_{k+1}y_{k+1})\\
=x_ky_{k+1}(1-x_{k+1})(1-y_k)
+x_{k+1}y_k(1-x_k)(1-y_{k+1})\\
+\big(x_kx_{k+1}+y_ky_{k+1}-(x_kx_{k+1})(y_ky_{k+1})\big).
\end{multline*}
For all $k$,
let $z_{k,k+1}^{(1)}=x_ky_{k+1}(1-x_{k+1})(1-y_k)$,
$z_{k,k+1}^{(2)}=x_{k+1}y_k(1-x_k)(1-y_{k+1})$,
and
$z_{k,k+1}^{(3)}=\big(x_kx_{k+1}+y_ky_{k+1}-(x_kx_{k+1})(y_ky_{k+1})\big),$
and let
\begin{align}\label{Equation: S 3}
S=\begin{cases}
\sum z_{1,2}^{(i_1)}z_{3,4}^{(i_3)}\cdots z_{K-1,K}^{(i_{K-1})}&\text{if $K$ is even}\\
\sum z_{1,2}^{(i_1)}z_{3,4}^{(i_3)}\cdots z_{K-2,K-1}^{(i_{K-2})}&\text{if $K$ is odd,}
\end{cases}
\end{align}
where the sum is over all $i_j\in\{1,2,3\}$ except $i_1=i_2=\cdots=3$.
By using the same arguments as in Lemma \ref{Lemma: Polynomial Expansion 4},
the fact that $x_k,y_k\geq x_ky_k$ for all $k\in \{1, \ldots , K\}$ implies the result.
\end{proof}

\begin{proposition}\label{Proposition: Polynomial Expansion 8}
Let $n\in\N$ and $0\leq x_1,\ldots,x_n,y_1,\ldots,y_n\leq1$.
Suppose that
$0<x_1,\ldots,x_n$,
or $0<y_1,\ldots,y_n$.
If
$$\prod_{k=1}^nx_k+\prod_{k=1}^ny_k-\prod_{k=1}^nx_ky_k=\prod_{k=1}^n(x_k+y_k-x_ky_k),$$
then for every distinct $k,l\leq K$
$$0=x_ky_l(1-x_l)(1-y_k).$$
\end{proposition}
\begin{proof}
The hypotheses of the present proposition implies that the sum $S$ defined in (\ref{Equation: S 3}),
is zero if $K$ is even,
or that $S(x_K+y_K-x_Ky_K)=0$ if $K$ is odd.
Given that $0\leq x,y\leq 1$,
we know that
$x+y-xy=0$ if and only if $x=y=0$.
This directly implies that $x_ky_{k+1}(1-x_k)(1-y_{k+1})=0$ for all odd $k\leq K-1$,
and result then follows by rearranging the index set as needed.
\end{proof}


\end{document}